\documentclass[11pt, final, reqno]{amsart}

\usepackage[utf8]{inputenc}
\usepackage[T1]{fontenc}
\usepackage[english]{babel}
\usepackage{amsmath,amsthm}
\usepackage{ae}
\usepackage{icomma}
\usepackage{units}
\usepackage{color}
\usepackage{graphicx}
\usepackage{bbm}
\usepackage{caption}
\usepackage{array}
\usepackage[hmarginratio=1:1]{geometry}
\usepackage[hyphens]{url}
\usepackage[pdfpagelabels=false,hidelinks]{hyperref}
\usepackage{mathrsfs}
\usepackage{amssymb}
\usepackage{placeins} 
\usepackage{tikz}
\usepackage{float} 
\usepackage[nodayofweek]{datetime}
\usepackage{enumitem}
\usepackage{mathtools}
\usepackage[numbers]{natbib}
\mathtoolsset{showonlyrefs=true}

\newcommand{\R}{\ensuremath{\mathbb{R}}}
\newcommand{\N}{\ensuremath{\mathbb{N}}}
\DeclareMathOperator{\dist}{\textnormal{dist}}

\DeclareMathOperator{\Tr}{Tr}
\renewcommand{\S}{\ensuremath{\mathbb{S}}}
\newcommand{\LT}[2]{L_{#1, #2}^{\textrm{cl}}}
\newcommand{\eps}{\ensuremath{\varepsilon}}
\newcommand{\grad}{\ensuremath{\nabla}}

\newcommand{\K}{\ensuremath{\mathcal{K}}}
\newcommand{\A}{\ensuremath{\mathcal{A}}}

\DeclareMathOperator{\supp}{supp}

\newcommand{\pps}{\hspace{0.75pt}}

\newtheorem{theorem}{Theorem}[section]

\newtheorem{lemma}[theorem]{Lemma}
\newtheorem{proposition}[theorem]{Proposition}
\newtheorem{corollary}[theorem]{Corollary}
\numberwithin{theorem}{section}
\numberwithin{definition}{section}
\theoremstyle{remark}
\newtheorem{remark}[theorem]{Remark}

\newcommand{\limplus}{{\mathchoice{\vcenter{\hbox{$\scriptstyle +$}}}
  {\vcenter{\hbox{$\scriptstyle +$}}}
  {\vcenter{\hbox{$\scriptscriptstyle +$}}}
  {\vcenter{\hbox{$\scriptscriptstyle +$}}}
}}
\newcommand{\limminus}{{\mathchoice{\vcenter{\hbox{$\scriptstyle -$}}}
  {\vcenter{\hbox{$\scriptstyle -$}}}
  {\vcenter{\hbox{$\scriptscriptstyle -$}}}
  {\vcenter{\hbox{$\scriptscriptstyle -$}}}
}}
\newcommand{\limpm}{{\mathchoice{\vcenter{\hbox{$\scriptstyle \pm$}}}
  {\vcenter{\hbox{$\scriptstyle \pm$}}}
  {\vcenter{\hbox{$\scriptscriptstyle \pm$}}}
  {\vcenter{\hbox{$\scriptscriptstyle \pm$}}}
}}

\newcommand{\specParam}{\ensuremath{\Lambda}}

\begin{document}

\title[Asymptotic shape optimization for Riesz means]{Asymptotic shape optimization for Riesz means of the Dirichlet Laplacian over convex domains}

\author[S. Larson]{Simon LARSON}
\address{\textnormal{(S. Larson)} Mathematical Sciences, Chalmers University of Technology and the University of Gothenburg, SE-41296 Gothenburg, Sweden}
\email{larsons@chalmers.se}


\begin{abstract}
For $\Omega \subset \R^n$, a convex and bounded domain, we study the spectrum of $-\Delta_\Omega$ the Dirichlet Laplacian on $\Omega$.
For $\specParam\geq0$ and $\gamma \geq 0$ let $\Omega_{\specParam, \gamma}(\A)$ denote any extremal set of the shape optimization problem
\begin{equation}
  \sup\{ \Tr(-\Delta_\Omega-\specParam)_\limminus^\gamma: \Omega \in \A, |\Omega|=1\}, 
\end{equation}
where $\A$ is an admissible family of convex domains in $\R^n$. If $\gamma \geq 1$ and $\{\specParam_j\}_{j\geq1}$ is a positive sequence tending to infinity we prove that $\{\Omega_{\specParam_j, \gamma}(\A)\}_{j\geq1}$ is a bounded sequence, and hence contains a convergent subsequence. Under an additional assumption on $\A$ we characterize the possible limits of such subsequences as minimizers of the perimeter among domains in $\A$ of unit measure. For instance if $\A$ is the set of all convex polygons with no more than $m$ faces, then $\Omega_{\specParam, \gamma}$ converges, up to rotation and translation, to the regular $m$-gon.

\medskip
\emph{This is a revised version of the paper published in the Journal of Spectral Theory (2019) which has been updated in accordance with an erratum published in 2021. The results of the paper remain unchanged, but the proofs of Theorem 2.4 and Corollary 5.3 have been amended.}
\end{abstract}

\maketitle


\section{Introduction and main results}
\subsection{Introduction}
This paper deals with the existence of an asymptotically optimal shape in a certain family of shape optimization problems. By a shape optimization problem we mean a variational problem where given a cost functional $\mathcal{F}$ and an admissible class of  domains $\mathcal{A}$ one wishes to solve the optimization problem 
\begin{equation}
  \inf\{ \mathcal{F}(\Omega): \Omega \in \A\}.
\end{equation}
For an introduction to the general theory of shape optimization we refer the reader to the books~\cite{MR2150214,MR2512810}.

In recent years the study of shape optimization for spectral problems, where the cost functional $\mathcal{F}$ depends on the spectrum of an operator defined on $\Omega$, has been of large interest, see for instance~\cite{Henrot_Book} and references therein. This type of problem has a long history which can be traced back to Lord Rayleigh~\cite{Rayleigh} who conjectured that the disk minimizes the first eigenvalue of the Dirichlet Laplacian among all planar domains of fixed area. Rayleigh's conjecture was proved independently by Faber~\cite{Faber} and Krahn~\cite{Krahn1}; 
the latter of whom also generalized the result to higher dimensions~\cite{Krahn2}.
From this result one can prove a similar statement concerning the second eigenvalue, namely that it is minimized by the union of two disjoint balls of equal measure~\cite{Krahn1, Krahn2, MR0061749}. For even higher eigenvalues the corresponding problems have only in recent years seen much progress. Using techniques coming from free boundary problems in partial differential equations it has been possible to prove the existence of extremal sets within the larger class of quasi-open sets\footnote{A quasi-open set is a superlevel set of a function in $H^1(\R^n)$, for a precise definition see~\cite{MR1217590}.} for the problem
\begin{equation}
   \inf\{ \lambda_k(\Omega): \Omega\subset \R^n \textrm{ quasi-open, } |\Omega|=1\}, 
\end{equation}
where $\lambda_k(\Omega)$ denotes the $k$-th eigenvalue of the Dirichlet Laplacian on $\Omega$ (see~\cite{MR2989451, MR3305655, MR1217590, MR3095209}). Within the same framework one can treat more general functionals depending on the eigenvalues of some spectral problem (see~\cite{MR3305655, MR3467381, MR3095209, MR3309888}).

Here we are interested in a two-parameter family of spectral shape optimization problems for the Dirichlet Laplacian, 
parametrized by $\gamma, \specParam\geq 0$ in~\eqref{eq:Riesz_mean_def} below. In the case $\gamma=0$ the problem essentially reduces to that of minimizing individual eigenvalues of the Dirichlet Laplacian but formulated in terms of the \emph{eigenvalue counting function}:
\begin{equation}\label{eq:def_counting_function}
  N_\Omega(\specParam):=\#\{k\in \N: \lambda_k(\Omega)<\specParam\}.
\end{equation}
Here we shall mainly consider the case $\gamma\geq 1$. 

For $\gamma\geq 1$ and $\specParam\geq 0$ the cost functionals we consider fit into the above mentioned framework for proving existence of extremal sets in the class of quasi-open sets of fixed measure. In the case $\gamma=1$ the problem is equivalent to that of minimizing the sum of the first $m$ eigenvalues for certain values of $m$, and thus it follows from~\cite{MR3305655,MazzTerrVel} that the optimal sets are open and their boundary is smooth up to exceptional sets of lower dimension. For $\gamma>1$ the question of whether the extremal sets are open is to the author's knowledge not covered by existing theory. However, this will not be the question dealt with in this paper. Instead, we restrict ourselves to the much simpler case of considering the problem when restricting the admissible class $\A$ to certain families of convex domains. Before we are able to properly define the functional considered it is necessary to introduce some additional notation.

Let $\Omega$ be an open subset of $\R^n$, $n\geq 2$, and let $-\Delta_\Omega$ denote the Dirichlet Laplace operator on $L^2(\Omega)$, which we define in the quadratic form sense with the Sobolev space $H_0^1(\Omega)$ as its form domain. If we assume that the measure of $\Omega$ is finite then the embedding $H_0^1(\Omega)\hookrightarrow L^2(\Omega)$ is compact, and hence the spectrum of $-\Delta_\Omega$ is discrete. Moreover, the spectrum consists of an infinite sequence of positive eigenvalues accumulating at infinity only. We enumerate these eigenvalues in an increasing sequence where each eigenvalue is repeated according to its multiplicity, 
\begin{equation}
  \lambda_1(\Omega)\leq\lambda_2(\Omega)\leq \lambda_3(\Omega)\leq \ldots 
\end{equation}

An open ball of radius $r>0$ centred at $x\in \R^n$ will be denoted by $B_r(x)$; if the centre of the ball is irrelevant we write simply $B_r$. For the ball of unit measure centred at the origin we write $B$.  

We can now define the two-parameter family of functionals studied here. For $\gamma\geq 0$ and $\specParam\geq0$ the \emph{Riesz eigenvalue means} of $-\Delta_\Omega$ are defined by
\begin{equation}\label{eq:Riesz_mean_def}
  \Tr(-\Delta_\Omega-\specParam)^\gamma_\limminus = \sum_{k: \lambda_k(\Omega)<\specParam} (\specParam- \lambda_k(\Omega))^\gamma, 
\end{equation}
where $x_\limpm := (|x|\pm x)/2$.

Given $\gamma \geq 0$, $\specParam\geq0$ and an admissible class of domains $\A$, we are interested in the shape optimization problem
\begin{equation}\label{eq:Riesz_shape_opt_prob}
    \sup\{ \Tr(-\Delta_\Omega-\specParam)^\gamma_\limminus: \Omega \in \A, |\Omega|=1\}.
\end{equation}
Here and in what follows we denote the $n$-dimensional measure of a set $\Omega\subset \R^n$ by $|\Omega|$ and the $(n-1)$-dimensional measure of its boundary by $|\partial\Omega|$.
For fixed $\gamma, \specParam$ and $\A$ let $\Omega_{\specParam, \gamma}(\A)$ denote any extremal domain of~\eqref{eq:Riesz_shape_opt_prob}. We emphasize that it is not a priori clear that any such domain exists. We shall here restrict our attention to $\gamma \geq 1$ and admissible classes $\A$ which are families of convex domains; without loss of generality we shall always assume that the admissible class $\A$ is closed under rigid transformations and contains at least one domain of unit measure. For such $\A$ the existence of extremal domains $\Omega_{\specParam, \gamma}(\A)$ will be proved in Lemma~\ref{lem:Existence_of_maximizer} below. 

We note that for $\gamma=0$ the Riesz mean is equal to the counting function of eigenvalues less than $\specParam$. Thus, in this case~\eqref{eq:Riesz_shape_opt_prob} is in a sense dual to the problem of minimizing $\lambda_k(\Omega)$. 

Moreover, the problem of maximizing the Riesz mean of order $\gamma=1$ is equivalent to minimizing the sum of the $m$ first eigenvalues for certain values of~$m$. Indeed, since $\Omega_{\specParam, 1}(\A)$ is extremal for~\eqref{eq:Riesz_shape_opt_prob} with $\gamma=1$, we have for any $\Omega\in \A$ with $|\Omega|=1$ that
\begin{equation}\label{eq:Equiv_to_sum1}
   \Tr(-\Delta_{\Omega_{\specParam,1}(\A)}-\specParam)_\limminus \geq \Tr(-\Delta_{\Omega}-\specParam)_\limminus.
\end{equation}
By definition~\eqref{eq:Equiv_to_sum1} is equivalent to
\begin{equation}
   \sum_{k\leq N_{\Omega_{\specParam,1}(\A)}(\specParam)}\lambda_k({\Omega_{\specParam,1}(\A)}) \leq \sum_{k\leq N_{\Omega}(\specParam)}\lambda_k(\Omega)-\specParam\bigl(N_{\Omega}(\specParam)-N_{\Omega_{\specParam,1}(\A)}(\specParam)\bigr).
\end{equation} 
We claim that the right-hand side is no larger than the sum of the $N_{\Omega_{\specParam,1}(\A)}(\specParam)$ first eigenvalues of $-\Delta_{\Omega}$. To this end let $m=N_{\Omega_{\specParam,1}(\A)}(\specParam)$ and $\widehat m = N_{\Omega}(\specParam)$. If $m=\widehat m$ the claim is clearly true. If $m<\widehat m$ then
\begin{align}
  \sum_{k\leq \widehat m} \lambda_k(\Omega)-\specParam(\widehat m- m) 
  = 
  \sum_{k\leq m}\lambda_k(\Omega)+\sum_{k = m+1}^{\widehat m}\lambda_k(\Omega)-\specParam(\widehat m -m)
  <
  \sum_{k\leq m}\lambda_k(\Omega),
\end{align}
where we in the last step used that $\lambda_k(\Omega)<\specParam$ for each $k\leq \widehat m$. The remaining case follows almost identically. Hence $\Omega_{\specParam, 1}(\A)$ is also extremal for the shape optimization problem
\begin{equation}
  \inf\Bigl\{\sum_{k=1}^{m}\lambda_k(\Omega): \Omega \in \A,\, |\Omega|=1\Bigr\},
\end{equation}
with $m=N_{\Omega_{\specParam, 1}(\A)}(\specParam).$

\subsection{Main results}
Let $\K^n$ denote the metric space defined as the set of all bounded convex domains $\Omega \subset \R^n$ with non-empty interior equipped with the Hausdorff distance~\cite{Schneider1}. We shall in this paper restrict our classes of admissible domains $\A$ to consisting of certain subsets of $\K^n$. In an upcoming paper it will be shown that these restrictions can be dropped~\cite{FrankLarson}. We begin by defining two natural classes of convex domains:
\begin{enumerate}[label=(\Alph*)]
  \item For an integer $m\geq n+1$ we let $\mathcal{P}_m\subset \K^n$ be the set of all bounded convex polytopes in $\R^n$ with no more than $m$ faces. We note that $\mathcal{P}_m$ is a closed subset of $\K^n$: If a sequence $\{P_j\}_{j\geq 1}\subset \mathcal{P}_m$ converges to $P\in \K^n$ in the topology of $\K^n$, then $P\in \mathcal{P}_m$. \\[-5pt]

  \item\label{itm:def_Komega} Fix a continuous increasing function $\omega\colon [0, L) \to \R$, with $\omega(0)=0$. Let $x\in \partial\Omega$, after rotation and translation we assume that $x=0$ and the hyperplane $\{x\in \R^n:x_n=0\}$ is tangent to $\partial\Omega$ at $x$. Let $D$ be the projection of $\partial\Omega \cap B_{L/2}$ onto this hyperplane. If $B_{L/2}\cap \partial\Omega$ can be represented as the graph of a function $f\in C^1(D)$ which satisfies
\begin{equation}
  |\nabla f(x')-\nabla f(y')|\leq \omega(|x'-y'|), \quad \forall x', y' \in D
\end{equation}
we say that $\partial\Omega$ has $C^1$-modulus of continuity $\omega$ around $x$. We say that $\partial\Omega$ is $\omega$-uniformly $C^1$ if this holds true with the same $\omega$ at every $x\in \partial\Omega$.

We let $\K^n_\omega$ denote the set of all $\Omega\in \K^n$ whose boundary is $\omega$-uniformly $C^1$. By the uniform regularity assumption it follows that also $\K_\omega^n$ is a closed subset of $\K^n$: If a sequence $\{K_j\}_{j\geq 1}\subset \K_\omega^n$ converges to $K\in \K^n$ in the topology of $\K^n$, then $K\in \K_\omega^n$. We shall always assume that $\omega$ is such that $\K^n_\omega$ contains at least one domain of unit measure.
\end{enumerate}

Our main results are contained in the following theorems:

\begin{theorem}\label{thm:Polytopes}
  Fix $\gamma \geq 1$ and $m\geq n+1$. Let $\{\specParam_j\}_{j\geq 1}\subset \R_\limplus$ be a sequence tending to infinity, and choose for each $j$ a corresponding extremal domain\/ $\Omega_j=\Omega_{\specParam_j, \gamma}(\mathcal{P}_m)$. 

  Then the sequence $\{\Omega_j\}_{j\geq 1}$ has a subsequence which, up to rigid transformations, converges to a domain $P_m\in \mathcal{P}_m$. Moreover, $P_m$ is of unit measure and minimizes the measure of the perimeter among domains in $\mathcal{P}_m$ of the same measure:
  \begin{equation}
    |\partial P_m|=\inf\{|\partial\Omega|: \Omega\in \mathcal{P}_m,\, |\Omega|=1\}.
  \end{equation}
\end{theorem}

We also prove the corresponding result in $\K^n_\omega$.
\begin{theorem}\label{thm:UnifC1}
  Fix  $\gamma \geq 1$ and $\omega$ as in~\ref{itm:def_Komega} above. Let $\{\specParam_j\}_{j\geq 1}\subset \R_\limplus$ be a sequence tending to infinity, and choose for each $j$ a corresponding extremal domain\/ $\Omega_j=\Omega_{\specParam_j, \gamma}(\K^n_\omega)$. 

  Then the sequence $\{\Omega_j\}_{j\geq 1}$ has a subsequence which, up to rigid transformations, converges to a domain $K_\omega \in \K^n_\omega$. Moreover, $K_\omega$ is of unit measure and minimizes the measure of the perimeter among domains in $\K^n_\omega$ of the same measure:
  \begin{equation}
    |\partial K_\omega|=\inf\{|\partial\Omega|: \Omega\in \K^n_\omega,\, |\Omega|=1\}.
  \end{equation}
\end{theorem}

We note that if $\A$ is one of the admissible classes considered above then the existence of a set $\Omega'$ realizing the infimum
\begin{equation}\label{eq:perimeter_minimizer}
   \inf\{|\partial\Omega|: \Omega\in \A,\, |\Omega|=1\}
\end{equation}
is an easy consequence of the strong compactness properties of $\K^n$~\cite{Schneider1}. If the set $\Omega'$ is unique, up to rigid transformations, then for any choice of sequence $\{\specParam_j\}_{j\geq 1}$ we find that the corresponding sequence of maximizers converges to $\Omega'$. Since the choice of sequence was arbitrary we obtain that
\begin{equation}\label{eq:Hausdorff_convergence}
   \inf_{\substack{x_0\in \R^n \\ T\in \mathrm{O}(n)}}\dist_{\K^n}(\Omega_{\specParam, \gamma}(\A), x_0+T\Omega') = o(1) \quad \mbox{as }\specParam \to \infty,
 \end{equation} 
 where $\mathrm{O}(n)$ is the orthogonal group in dimension $n$ and $\dist_{\K^n}$ denotes the metric of~$\K^n$.
Since we do not know that the maximizers $\Omega_{\specParam, \gamma}(\A)$ are unique we emphasize that we mean that~\eqref{eq:Hausdorff_convergence} is true when an arbitrary choice of maximizer is made for each $\specParam$.

In particular if $\omega$ is such that the unit ball $B\in \K_\omega^n$ then it is up to translations the unique minimizer of~\eqref{eq:perimeter_minimizer} and hence $\Omega_{\specParam, \gamma}(\K_\omega^n)$ converges to $B$, modulo translations. If the ball is not in $\K_\omega^n$ then minimizers of the perimeter need not be unique and different subsequences of $\Omega_{\specParam_j,\gamma}(\K^n_\omega)$ may converge to different such minimizers.

The existence and characterization of minimizers of the perimeter in the class $\mathcal{P}_m$ is a classical problem. This problem is equivalent to that of finding which polytopes circumscribing a ball have minimal volume~\cite{Gruber}. For $n=2$ the regular $m$-gon is, up to rotations and translations, the unique minimizer. However, in higher dimensions this turns out to be a very difficult problem, and to the author's knowledge it is not known whether the minimizers are unique.

If $\Omega'$ realizing~\eqref{eq:perimeter_minimizer} is not unique then one can still conclude that all isolated minimizers of the perimeter are local asymptotic maximizers of our shape optimization problem in the following sense: Let $\Omega'\in \A$ realize the infimum~\eqref{eq:perimeter_minimizer} and assume that $\Omega'$ is isolated from any other such minimizer with respect to the Hausdorff topology (up to rigid transformations). Then one can construct a perturbed shape optimization problem by removing from $\A$ an arbitrarily small neighbourhood around all other minimizers of the perimeter (in the Hausdorff sense). For this new shape optimization problem any sequence of maximizers would converge to the now unique minimizer of the perimeter. 

\subsection{Related results and further questions}
Similar results in asymptotics of extremal domains have recently been obtained in several different settings. The most commonly studied problem is that of finding a domain asymptotically minimizing $\lambda_k(\Omega)$ among $\Omega$ in a certain class of admissible domains. That is, given an admissible class of domains $\mathcal{A}$ one wants to find a domain $\Omega_\infty$ such that the extremizers of the problem 
\begin{equation}\label{eq:shape_opt_eigenvalues}
\inf\{\lambda_k(\Omega): \Omega\in \mathcal{A}\}
\end{equation} 
converge to $\Omega_\infty$ as $k$ goes to infinity. 

The first result in this direction is due to Antunes and Freitas who proved that if $\mathcal{A}$ is the set of rectangles with area one, then any sequence of extremal sets converges to the unit square as $k$ goes to infinity~\cite{MR3001382}. In~\cite{vdBergGittins} van den Berg and Gittins proved the corresponding result in three dimensions, and in~\cite{GittinsLarson} the result was obtained in general dimension: In the class $\mathcal{A}$ of sets of the form $(0, a_1) \times \ldots \times (0, a_n)\subset \R^n$ of unit measure any sequence of minimizers of the $k$-th eigenvalue converges to the unit cube in $\R^n$ as $k\to \infty$. In~\cite{vdBBG,GittinsLarson} the corresponding results were proved to hold also if one instead considers eigenvalues of the Neumann Laplacian on the same class of domains, in which case the natural problem is to maximize the eigenvalues.

The idea of Antunes and Freitas~\cite{MR3001382} was to reformulate the problem of minimizing eigenvalues as a maximization problem for the counting function~\eqref{eq:def_counting_function} and exploit the explicit structure of Laplacian eigenvalues on rectangles. This effectively reformulates the problem as an optimization problem in the setting of geometric lattice point counting: For fixed $\specParam\geq 0$ find the ellipses among those on the form $(x/a)^2+(a y)^2 \leq \specParam/\pi^2$ which contain the greatest number of positive integer lattice points. The asymptotic problem translates into studying the shape of such ellipses in the limit $\specParam\to \infty$.

The lattice point problem which arose in the work of Antunes and Freitas has since then seen several generalizations. In~\cite{LaugesenLiu1,AriturkLaugesen} 
Laugesen and Liu resp.\ Ariturk and Laugesen
consider a similar problem but replace the bounding region, which in~\cite{MR3001382} was given by a quarter of an ellipse, by the region under the graph of a decreasing concave or convex function~$f$. The optimization problem studied can then be phrased as follows: Given $r>0$ find $s>0$ realizing
\begin{equation}
  \sup_{s>0} \#\{(j, k)\in \N^2: k\leq r s f(j s/r)\}.
\end{equation}  
The main results of~\cite{AriturkLaugesen,LaugesenLiu1} are that under weak assumptions on $f$ the optimal values of $s$ tend to a unique limit as $r\to \infty$. Moreover, the limit can be explicitly expressed in terms of~$f$. More recently these results have been generalized to allow for a shift of the lattice, that is replacing the standard lattice by $(\N+\sigma)\times(\N+\tau)$, see~\cite{LaugesenLiu2}. Also higher dimensional versions of this problem have been studied by Marshall, and Guo and Wang in~\cite{GuoWang,Marshall}.

A particularly interesting case of the lattice point optimization problem is to consider $f(x)=1-x$. In this case the behaviour of maximizing values $s$ is highly erratic, and it was proved by Marshall and Steinerberger~\cite{MarshallSteinerberger} that there is no unique limit as $r$ tends to infinity. In fact they prove that there are infinitely many values of $s$ which are optimal for arbitrarily large $r$, which proves a conjecture of Laugesen and Liu in~\cite{LaugesenLiu1}. Recently a related problem but in the setting of Riesz means was studied in~\cite{LarsonHarmonicOscillators}.

In the same direction as the work of Antunes and Freitas, one can consider the shape optimization problem~\eqref{eq:shape_opt_eigenvalues} as $k$ tends to infinity but with the measure constraint replaced by different ones, see~\cite{MR3312973, MR3098942,FreitasAverages}.
In particular, Bucur and Freitas considered the problem in $\R^2$ under a constraint on the measure of the perimeter and prove that any sequence of optimal domains converges to the disk~\cite{MR3098942}: If $\Omega_k\subset \R^2$ is a domain realizing the infimum
\begin{equation}
   \inf\{\lambda_k(\Omega): \Omega \subset \R^2 \mbox{ open}, |\partial\Omega|\leq |\partial B|\},
\end{equation}
then, up to translation, $\lim_{k\to \infty}\Omega_k = B$ in the Hausdorff topology. In~\cite{FreitasAverages} Freitas considered the problem of minimizing the average (or equivalently the sum) of the $m$ first eigenvalues in the limit as $m$ tends to infinity under a constraint on either the measure or the perimeter. In the former case he obtains that the extremal averages are in a certain sense sub-additive and compute their leading order asymptotic behaviour. In the latter he proves that the extremal sets converge to a ball in the limit $m\to \infty$.

The fact that the problem studied here allows the same type of analysis as in the results discussed above under the constraint of fixed measure, and in large classes of convex sets is the main reason that we find it noteworthy. Moreover, after this article was completed it has been proved that the a priori regularity assumptions on $\A$ needed to identify the asymptotically maximizing domains as minimizers of the perimeter can be removed. That these assumptions can be dropped is a consequence of the results in~\cite{FrankLarson} where two-term asymptotic formulas for $\Tr(-\Delta_\Omega-\Lambda)_\limminus$ are obtained in the semi-classical limit $\Lambda\to \infty$, under the assumption that $\Omega\subset\R^n$ has Lipschitz-regular boundary (in particular this covers all convex domains).

A natural further question is of course whether the convexity assumption can be dropped, and instead consider the optimization problem~\eqref{eq:Riesz_shape_opt_prob} in the class of quasi-open sets. As mentioned in the introduction the existence of optimizers for this problem with $\gamma\geq 1$ is covered by the results in~\cite{MR2989451, MR1217590, MR3506069, MR3095209} (see also~\cite{MR3467381, MR3309888}). The results in these articles consider the case of functions of a \emph{fixed} number of eigenvalues which is not the case for~\eqref{eq:Riesz_mean_def}. However, using the Li--Yau inequality~\cite{MR701919},
\begin{equation}\label{eq:LiYau}
   \lambda_k(\Omega) \geq \Gamma\Bigl(\frac n2 +1 \Bigr)^{2/n}\frac{4\pi n}{n+2}\Bigl(\frac{k}{|\Omega|}\Bigr)^{2/n},
\end{equation} 
we can bound the number of eigenvalues present in the sum~\eqref{eq:Riesz_mean_def} and thus reduce our problem to this situation. 
If $\specParam > \lambda_1(B)$ and $\gamma \geq 1$ then the functional $\Tr(-\Delta_\Omega-\specParam)^\gamma_\limminus$ is Lipschitz continuous as a function of the eigenvalues and weakly strictly decreasing in a neighbourhood of any maximizer (see~\cite{MR3506069}). Hence the problem is covered by the existing results for such cost functions. 
However, in terms of what happens as $\specParam \to \infty$ these results yield little information. To analyse the asymptotic behaviour of maximizers we here use inequalities for $\Tr(-\Delta_\Omega-\specParam)_\limminus^\gamma$ in terms of geometric quantities of $\Omega$, see Theorem~\ref{thm:Improved_Berezin_Convex}. Similar inequalities have in recent years been obtained in a variety of different forms, see e.g.~\cite{GeisingerLaptevWeidl,HarrellStubbe_16,HarrellProvenzanoStubbe_18,WeidlKovarikVugalter,Melas,Weidl08}. These inequalities indeed provide geometric information about maximizers in a more general setting. However, without the convexity assumption it is unclear whether this information is sufficient to prove that maximizers cannot degenerate as $\specParam \to \infty$.

\subsection{Structure of the paper} The remainder of the paper is structured as follows. In Section~\ref{sec:Prel} we introduce some notation, recall some known results, and prove a number of inequalities needed in the sequel. Section~\ref{sec:Existence} is devoted to proving that given an admissible class of convex domains $\A$ the shape optimization problem~\eqref{eq:Riesz_shape_opt_prob} has at least one extremal domain for fixed values of $\specParam$ and $\gamma$. 
In Section~\ref{sec:Convergence} we establish that for $\gamma\geq 1$ any sequence of extremal domains has a convergent subsequence, and show that under an additional assumption on the class $\A$ any limit point of the sequence must be a minimizer of the perimeter. 
In Section~\ref{sec:sums} we show that the tools developed to prove our main theorems also allow us to deduce the corresponding results when minimizing the sum of the first $m$ eigenvalues among convex domains. Section~\ref{sec:Asymptotics} is devoted to studying the asymptotic behaviour of~\eqref{eq:Riesz_mean_def} as $\specParam\to \infty$, and proving that the assumption from Section~\ref{sec:Convergence} holds true in~$\mathcal{P}_m$. That the same assumption is true in~$\K^n_\omega$ is a consequence of the results in~\cite{MR2885166,MR2994205} (see Lemma~\ref{lem:Asymptotic_Komega}). We end the paper by proving that our results generalize to the case when the admissible domains are allowed to consist of disjoint unions of convex domains, see Section~\ref{sec:Disjoint_unions}.


\section{Notation and preliminaries}\label{sec:Prel}

We denote by $\dist(\, \cdot \,, \cdot\,)$ the distance between two sets in $\R^n$ (possibly singletons):
\begin{equation}
  \dist(A, B) := \inf_{x\in A, y\in B}|x-y|.
\end{equation}
We will in several places make use of the fact that if $\partial\Omega$ is Lipschitz regular then $\dist(\,\cdot\,, \partial\Omega)$ satisfis
\begin{equation}\label{eq:Eikonal}
  |\grad\! \dist(x, \partial\Omega)|=1
\end{equation}
for almost every $x\in \R^n$. In particular this holds true as soon as $\Omega$ is convex.

\subsection{Preliminary convex geometry}\label{sec:Prel_convex}
We continue by recalling some basic definitions from convex geometry and introducing the notation we use.
For more details and a general treatment of classical convex geometry we refer the reader to the books~\cite{Gruber,Schneider1}.

Let $\Omega \in \K^n$. We define the \emph{inradius}, \emph{diameter}, and \emph{minimal width} of $\Omega$ by
\begin{equation}
  r(\Omega):=\sup_{x\in \Omega}\dist(x, \Omega^c), \quad D(\Omega):=\sup_{x, y\in \Omega}|x-y|, \quad \mbox{resp.} \quad w(\Omega):= \inf_{\nu\in\S^{n-1}}\bigl(\sup_{x\in\Omega} x\cdot \nu-\inf_{x\in\Omega\vphantom{\tilde \Omega}}x\cdot \nu\bigr).
\end{equation}
We note that $r$ is the radius of the largest ball contained in $\Omega$, and $w$ is the smallest distance such that $\Omega$ is contained between two parallel hyperplanes separated by this distance.

Clearly $2r(\Omega)\leq w(\Omega)$. Less clear is that also a reversed inequality holds~\cite{Steinhagen}: There exists a dimensional constant $c>0$ such that, for all $\Omega \in \K^n$,
\begin{equation}\label{eq:Steinhagen}
  c w(\Omega) \leq r(\Omega).
\end{equation}

The \emph{inner} and \emph{outer parallel sets} of $\Omega$ at distance $t\in (0, \infty)$ are defined by
\begin{align}
   \Omega_t &:= \Omega \sim B_t = \{x\in \Omega: \dist(x, \Omega^c)>t\},\\
   \Omega^t &:= \Omega + B_t = \{x \in \R^n: \dist(x, \Omega)<t\}.
 \end{align} 
The notation $+, \sim$ comes from the concepts of Minkowski addition and the Minkowski difference~\cite{Schneider1}.

We let $W\colon (\K^n)^n \to \R$ denote the unique symmetric function (with respect to permutations of its arguments) such that
\begin{equation}\label{eq:Steinerformula}
  |\eta_1\Omega_1 + \ldots + \eta_m\Omega_m|= \sum_{j_1=1}^m\cdots \sum_{j_n=1}^m \eta_{j_1}\!\cdots \eta_{j_n}W(\Omega_{j_1}, \ldots, \Omega_{j_n}),
\end{equation}
for any $\Omega_{1}, \ldots, \Omega_m \in \K^n$ and $\eta_1, \ldots, \eta_m\geq 0$~\cite{Schneider1}. The quantity $W(\Omega_1, \ldots, \Omega_n)$ is called the \emph{mixed volume} of $\Omega_1, \ldots, \Omega_n \in \K^n$. Here we will only need the following elementary properties of~$W$ (see~\cite{Schneider1}):
\begin{enumerate}[label=\arabic*.]
  \item $W(\Omega_1, \ldots, \Omega_n)>0$ for $\Omega_1, \ldots, \Omega_n \in \K^n$.
  \item $W$ is a multilinear function with respect to Minkowski addition.
  \item $W$ is increasing with respect to inclusions.
  \item $W$ is invariant under translations in each argument.
  \item The volume and perimeter of $\Omega\in \K^n$ can be written in terms of $W$:
  \begin{equation}
    |\Omega|=W(\Omega, \ldots, \Omega) \quad \mbox{and} \quad  |\partial\Omega|= n W(\Omega, \ldots, \Omega, B_1).
  \end{equation}
\end{enumerate}
 
In what follows we shall need to bound certain of these quantities in terms of others; these and similar bounds can be found in the literature but we include proofs for completeness. To this end we recall the main result of~\cite{LarsonJFA}: For $t\geq 0$ and any $\Omega \in \K^n$ it holds that
\begin{equation}\label{eq:bound_per_innerparallel}
  |\partial\Omega_t|\geq |\partial\Omega|\Bigl(1- \frac{t}{r(\Omega)}\Bigr)^{n-1}_\limplus.
\end{equation}
Since the measure of the perimeter of convex sets is decreasing under set inclusion we also have that $|\partial\Omega_t|\leq |\partial\Omega|$.

Using the co-area formula and~\eqref{eq:Eikonal} we have that
\begin{equation}
  |\Omega|= \int_{0}^{r(\Omega)}|\partial\Omega_t|\,dt.
\end{equation}
By the upper, respectively lower, bound on $|\partial\Omega_t|$ above we find, after integrating and rearranging, that
\begin{equation}\label{eq:twosided_inradius_bound}
  \frac{|\Omega|}{|\partial\Omega|}\leq r(\Omega)\leq n\frac{|\Omega|}{|\partial\Omega|}.
\end{equation}

Furthermore, it is not difficult to deduce an upper bound for $D(\Omega)$ in terms of $r(\Omega)$ and~$|\Omega|$: After translation and rotation we may assume that the ball $B_{r(\Omega)}(0)\subset \Omega$ and that $x_0=(0, \ldots, 0, R)\in\Omega$. By convexity the cone $V$ with vertex $x_0$ and base $\{x\in \R^n: x_n=0,\ x_1^2+\ldots + x_{n-1}^2= r(\Omega)^2\}$ is contained in $\Omega$. The volume of this cone is equal to
\begin{equation}
  |V| = c r(\Omega)^{n-1} \int_{0}^R \Bigl(1- \frac{x_n}{R}\Bigr)^{n-1}\,dx_n = c r(\Omega)^{n-1}R.
\end{equation}
Thus we have a contradiction if $c r(\Omega)^{n-1}R \geq |\Omega|$ and hence $R\leq c \frac{|\Omega|}{r(\Omega)^{n-1}}$. Consequently there is a constant $c>0$, depending only on $n$, such that
\begin{equation}\label{eq:diameter_bound}
  D(\Omega) \leq c \frac{|\Omega|}{r(\Omega)^{n-1}}.
\end{equation}

\subsection{Weyl asymptotics}\label{sec:Prel_Weyl}
From the classical Weyl asymptotics for the Dirichlet eigenvalues (see~\cite{WeylAsymptotic}) it follows that the Riesz means for $\gamma\geq 0$ obey the asymptotic formula
\begin{equation}
   \Tr(-\Delta_\Omega-\specParam)^\gamma_\limminus = \LT{\gamma}{n} |\Omega| \specParam^{\gamma+n/2} + o(\specParam^{\gamma+n/2}) \quad \mbox{as }\specParam \to \infty.
\end{equation} 
Here $\Omega\subset \R^n$ is a bounded and open set and $\LT{\gamma}{n}$ denotes the semi-classical Lieb--Thirring constant:
\begin{equation}
  \LT{\gamma}{n} = \frac{\Gamma(\gamma+1)}{(4\pi)^{n/2}\Gamma(\gamma+1+n/2)}.
\end{equation}

If in addition $\Omega$ satisfies certain regularity properties the following two-term asymptotic formula holds:
\begin{equation}\label{eq:Two_Term_Weyl}
  \Tr(-\Delta_\Omega-\specParam)^\gamma_\limminus = \LT{\gamma}{n} |\Omega| \specParam^{\gamma+n/2} - \frac{\LT{\gamma}{n-1}}{4}|\partial\Omega|\specParam^{\gamma+(n-1)/2} +  o(\specParam^{\gamma+(n-1)/2}), 
\end{equation}
as $\specParam \to \infty$. This refined asymptotic formula was conjectured by Weyl in~\cite{WeylAsymptotic}. 

Under the sole assumption of convexity we prove that the asymptotic behaviour does not lie below that suggested by the Weyl conjecture:
\begin{lemma}[One-sided two-term asymptotics]\label{lem:one_sided_two_term_asymptotics}
  Let\/ $\Omega\in \K^n$. Then, for $\gamma \geq 1$, 
  \begin{equation}\label{eq:lem_onesided_asymptotics}
    \Tr(-\Delta_\Omega-\specParam)^\gamma_\limminus \geq \LT{\gamma}{n} |\Omega| \specParam^{\gamma+n/2} - \frac{1}{4}\LT{\gamma}{n-1}|\partial\Omega|\specParam^{\gamma+(n-1)/2} +  o(\specParam^{\gamma+(n-1)/2}), 
\end{equation}
as $\specParam \to \infty$. Moreover, the error term is uniform on compact subsets of $\K^n$.
\end{lemma}

In~\cite{MR575202} Ivrii proved that~\eqref{eq:Two_Term_Weyl} holds for $\gamma=0$ under the assumptions that $\partial\Omega$ is smooth and the measure of the periodic billiards in $\Omega$ is zero. By the Aizenman--Lieb identity it follows that the expansion holds for all $\gamma>0$ under the same assumptions. More recently, Frank and Geisinger proved that~\eqref{eq:Two_Term_Weyl} is true for $\gamma=1$ if the boundary of $\Omega$ is $C^{1, \alpha}$-regular~\cite{MR2885166}. In~\cite{MR2994205} the same authors treat the case of Robin boundary conditions and show that their method also covers $C^1$-domains. Again the Aizenman--Lieb identity implies that~\eqref{eq:Two_Term_Weyl} is valid under the same assumptions for all $\gamma\geq 1$. In particular, the results of~\cite{MR2885166,MR2994205} imply that the expansion~\eqref{eq:Two_Term_Weyl} holds uniformly on compact subsets of~$\K_\omega^n$.

\begin{lemma}[{\cite[Theorem~1.1]{MR2885166}, \cite[Theorem~1.3]{MR2994205}}]\label{lem:Asymptotic_Komega}
  Let\/ $\Omega\in \K^n_\omega$. Then, for $\gamma \geq 1$,
  \begin{equation}
     \Tr(-\Delta_\Omega-\specParam)^\gamma_\limminus = \LT{\gamma}n |\Omega| \specParam^{\gamma+n/2}-\frac{\LT{\gamma}{n-1}}{4}|\partial \Omega| \specParam^{\gamma+(n-1)/2}+o(\specParam^{\gamma+(n-1)/2}),
  \end{equation}
  as $\specParam \to \infty$. Moreover, the error term is uniform on compact subsets of $\K_\omega^n$.
\end{lemma}

That the error term in the above expansion is uniform on compact subsets follows from the methods of Frank and Geisinger, in fact the uniform $C^1$-modulus of continuity of $\partial\Omega$ together with upper and lower bounds on $|\Omega|$ and $|\partial\Omega|$ suffices. This uniformity is not explicitly stated in their results but it is nonetheless possible to track the geometric dependence through their proof and conclude that this is the case. However, this is not an entirely trivial task. To see how this can be done we refer the reader to~\cite{FrankLarson} where the same construction is used and the error term is tracked explicitly.

In Section~\ref{sec:Asymptotics} we shall prove that~\eqref{eq:Two_Term_Weyl} holds uniformly also for $\Omega$ in compact subsets of~$\mathcal{P}_m$.
\begin{lemma}\label{lem:Asymptotic_Polytopes}
  Let\/ $\Omega\in \mathcal{P}_m$. Then, for $\gamma \geq 1$,
  \begin{equation}
     \Tr(-\Delta_\Omega-\specParam)^\gamma_\limminus = \LT{\gamma}n |\Omega| \specParam^{\gamma+n/2}-\frac{\LT{\gamma}{n-1}}{4}|\partial \Omega| \specParam^{\gamma+(n-1)/2}+o(\specParam^{\gamma+(n-1)/2}),
  \end{equation}
  as $\specParam \to \infty$. Moreover, the error term is uniform on compact subsets of $\mathcal{P}_m$.
\end{lemma}

The reason that we here need to further restrict our admissible classes of convex domains is that, prior to the results in~\cite{FrankLarson},~\eqref{eq:Two_Term_Weyl} was not known to hold uniformly in compact subsets of $\K^n$, for $\gamma \geq 1$. 

The refined asymptotics~\eqref{eq:Two_Term_Weyl} combined with the isoperimetric inequality indicates that if we can prove that an asymptotically optimal shape exists, it is likely the ball. This is indeed the heuristic idea behind the belief that maximizers of the Riesz means, or for that matter minimizers of the eigenvalues, should be well behaved in the limit $\specParam\to \infty$.

\subsection{A two-term Berezin inequality}\label{sec:Prel_Berezin}
A key ingredient in our proof here will be the following two-term bound for the Riesz means of order $\gamma\geq 1$ when $\Omega\subset \R^n$ is convex. This result was first obtained for $\gamma\geq 3/2$ in the planar case in~\cite{GeisingerLaptevWeidl} under an additional geometric assumption. In~\cite{LarsonJFA} it was proved that this additional assumption was true in general, and in~\cite{LarsonPAMS} this was used to generalize the bound for $\gamma\geq 3/2$ to any dimension and arbitrary convex domains. As we shall see, the extension to $1\leq \gamma<3/2$ can be obtained by combining results of~\cite{FrankLarson} and~\cite{KovarikWeidl}. 

\begin{theorem}\label{thm:Improved_Berezin_Convex}
Let\/ $\Omega\in \K^n$. For $\gamma\geq 1$ there exists a constant $c(\gamma, n)>0$ such that
\begin{align}
  \textrm{if } \specParam \leq \frac{\pi^2}{4 r(\Omega)^2}: \quad\Tr(-\Delta_\Omega-\specParam)^\gamma_\limminus =0, \hspace{208.5pt}\\
  \textrm{if } \specParam > \frac{\pi^2}{4 r(\Omega)^2}:\quad  \Tr(-\Delta_\Omega-\specParam)^\gamma_\limminus \leq \LT{\gamma}{n} |\Omega| \specParam^{\gamma+n/2} - c(\gamma, n) \LT{\gamma}{n-1} |\partial \Omega| \specParam^{\gamma+(n-1)/2}. 
  \end{align}
\end{theorem}

\begin{proof}[Proof of Theorem~\ref{thm:Improved_Berezin_Convex}]
The first part of the theorem is a direct consequence of the bound $\lambda_1(\Omega)\geq \frac{\pi^2}{4r(\Omega)^2}$ proved in \cite{Hersch,MR589137}. For the second part, Theorem 1.2 in \cite{FrankLarson} implies that there exist universal constants $c_1, c_2>0$ such that
\begin{equation}\label{eq: frank larson bound}
  \Tr(-\Delta_\Omega-\Lambda)_\limminus \leq \LT{1}{n}|\Omega|\Lambda^{1+n/2}- c_2\LT{1}{n-1}|\partial\Omega|\Lambda^{1+(n-1)/2}, \quad \mbox{for all }\Lambda \geq \frac{c_1}{r(\Omega)^2}.
\end{equation}
In order to extend~\eqref{eq: frank larson bound} to all 
$\Lambda \geq \frac{\pi^2}{4r(\Omega)^2}$ we recall Theorem~3.5 in~\cite{KovarikWeidl} which for $\Omega \in \K^n$ implies that
\begin{equation}
  \Tr(-\Delta_\Omega-\Lambda)_\limminus\leq \LT{1}n |\Omega|\Lambda^{1+n/2}- \frac{\LT{1}n}{72}r(\Omega)^{-3/2}|\Omega|\Lambda^{n/2+1/4}, \quad \mbox{for all }\Lambda \geq \frac{\pi^2}{4r(\Omega)^2}.
\end{equation}
Note that this is almost what we aim to prove but the negative correction term on the right-hand side differs from that in~\eqref{eq: frank larson bound} both in terms of the dependence on $\Omega$ and in the power of $\Lambda$. However, provided the quantity $r(\Omega)^2\Lambda$ remains bounded the two correction terms can be related as follows.

We claim that there is a universal constant $c_3>0$ such that
\begin{equation*}
  c_3\LT{1}{n-1}|\partial\Omega|\Lambda^{1+(n-1)/2}\leq \frac{\LT{1}n}{72}r(\Omega)^{-3/2}|\Omega|\Lambda^{n/2+1/4}\quad \mbox{for } \frac{\pi^2}{4r(\Omega)^2}\leq \Lambda \leq \frac{c_1}{r(\Omega)^2},
\end{equation*}
which allows us to extend~\eqref{eq: frank larson bound} to all $\Lambda \geq \frac{\pi^2}{4r(\Omega)^2}$ (possibly with a smaller constant $c_2$). The existence of such a constant follows from 
\begin{equation}
  \inf_{\Omega \in \K^n} \inf_{\frac{\pi^2}{4r(\Omega)^2}\leq \Lambda \leq \frac{c_1}{r(\Omega)^2}} \frac{|\Omega|}{|\partial\Omega|r(\Omega)^{3/2}}\Lambda^{-1/4} = \inf_{\Omega \in \K^n} \frac{|\Omega|c_1^{-1/4}}{|\partial\Omega|r(\Omega)} \geq \frac{1}{nc_1^{1/4}},
\end{equation}
where we in the second step used~\eqref{eq:twosided_inradius_bound}. This completes the proof of Theorem~\ref{thm:Improved_Berezin_Convex} for $\gamma=1$.

To prove the case $\gamma>1$ we apply the Aizenman--Lieb identity~\cite{AizenmanLieb} (see also~Section~\ref{sec:Asymptotics}). For all $\Lambda \geq 0$ and $b\in [0, 1]$ the case $\gamma =1$ implies that
\begin{align*}
  \Tr(-\Delta_\Omega -\Lambda)_\limminus 
  &\leq 
  \Bigl(\LT{1}n |\Omega|\Lambda^{1+n/2}- c\LT{1}{n-1}|\partial\Omega|\Lambda^{1+(n-1)/2}\Bigr)_\limplus\\
  &\leq
  \LT{1}n|\Omega|\Lambda^{1+n/2}\Bigl(1- \frac{bc \LT{1}{n-1}}{\LT1n}\frac{|\partial\Omega|}{|\Omega|}\Lambda^{-1/2}+\frac{b^2c^2(\LT{1}{n-1})^2}{4(\LT1n)^2}\frac{|\partial\Omega|^2}{|\Omega|^2}\Lambda^{-1}\Bigr)_\limplus\\
  &=
  \LT{1}n|\Omega|\Lambda^{1+n/2}\Bigl(1- \frac{bc \LT{1}{n-1}}{2\LT1n}\frac{|\partial\Omega|}{|\Omega|}\Lambda^{-1/2}\Bigr)^2\\
  &=
  \LT1n |\Omega|\Lambda^{1+n/2}- c_1 b \LT{1}{n-1}|\partial\Omega|\Lambda^{1+(n-1)/2}+ c_2b^2 \LT{1}{n-2}\frac{|\partial\Omega|^2}{|\Omega|}\Lambda^{1+(n-2)/2},
\end{align*}
for some positive constants $c_1, c_2$. Here the positive part in the first inequality is necessary as we do not distinguish between the two cases of Theorem~\ref{thm:Improved_Berezin_Convex}. An application of the Aizenman--Lieb identity yields
\begin{equation}
  \Tr(-\Delta_\Omega -\Lambda)_\limminus^\gamma
  \leq
  \LT{\gamma}n |\Omega|\Lambda^{\gamma+n/2}- c_1 b \LT{\gamma}{n-1}|\partial\Omega|\Lambda^{\gamma+(n-1)/2}+ c_2b^2 \LT{\gamma}{n-2}\frac{|\partial\Omega|^2}{|\Omega|}\Lambda^{\gamma+(n-2)/2},
\end{equation}
for all $\Lambda \geq 0$ and $b\in [0, 1]$. From~\eqref{eq:twosided_inradius_bound} we have $\frac{|\partial\Omega|}{|\Omega|}\Lambda^{-1/2}\leq \frac{2n}{\pi}$ for all $\Lambda \geq \frac{\pi^2}{4r(\Omega)^2}$ and thus the claimed bound follows by choosing $b$ sufficiently small.
\end{proof}

\begin{remark}
  In the published version of this paper the proof of Theorem~\ref{thm:Improved_Berezin_Convex} was based on combining results in~\cite{HarrellStubbe_16} and~\cite{GittinsLarson}. However, the proof of the relevant inequality in~\cite{HarrellStubbe_16} was incorrect (see~\cite{HarrellStubbe_Retraction}) which resulted in a gap in our proof of Theorem~\ref{thm:Improved_Berezin_Convex} and necessitated the publication of the erratum~\cite{LarsonErratum}.
\end{remark}

The bound in Theorem~\ref{thm:Improved_Berezin_Convex} is an improvement of an inequality going back to Berezin~\cite{Berezin} which states that for the convex Riesz means, i.e.\ when $\gamma\geq 1$, the first term in the Weyl asymptotic formula always overestimates the eigenvalue mean:
\begin{equation}\label{eq:BerezinIneq}
  \Tr(-\Delta_\Omega-\specParam)^\gamma_\limminus \leq \LT{\gamma}{n}|\Omega|\specParam^{\gamma+n/2}. 
\end{equation}
This inequality should more correctly be attributed to Berezin, Lieb and Li--Yau~\cite{Berezin,MR701919,Lieb}: Berezin and Lieb both proved inequalities of which~\eqref{eq:BerezinIneq} is a special case, while Li and Yau proved an inequality for the sum of the first~$m$ eigenvalues which is equivalent to~\eqref{eq:BerezinIneq} (see~\cite{Laptev_97}). 

We emphasize that the second term appearing in the bound of Theorem~\ref{thm:Improved_Berezin_Convex} is up to a constant the same as that appearing in the refined Weyl asymptotic formula (this is essential in proving the boundedness of the maximizers).


\section{Existence of extremal domain}\label{sec:Existence}
For any fixed $\gamma\geq 1$ and $\specParam$ large enough, we have that the existence of a maximizer in the class of quasi-open sets follows from known results~\cite{MR2989451, MR1217590,  MR3506069, MR3095209}. However, the methods used in these articles do not take into account that we wish to stay within our class of convex domains. But, as this is already a very nice class of sets, proving the existence of a maximizer for our problem is not difficult. 

\begin{lemma}[Existence of maximizers]\label{lem:Existence_of_maximizer}
  Let $\A$ be a closed subset of $\K^n$. Then, for any $\gamma \geq 0$ and $\specParam \geq 0$ there exists a domain $\Omega_{\specParam, \gamma}\in \A$ realizing the supremum
  \begin{equation}\label{eq:MaxProb_Existence_thm}
    \sup\{ \Tr(-\Delta_\Omega-\specParam)^\gamma_\limminus : \Omega \in \A, |\Omega|=1\}.
  \end{equation}
   Moreover, if $\A=\K^n$, $\gamma\geq 1$, and $\specParam > \lambda_1(B)$ then any such domain has $C^1$-regular boundary.
\end{lemma}

\begin{proof}[Proof of Lemma~\ref{lem:Existence_of_maximizer}]
  For fixed $\specParam>\lambda_1(B)$ and $\gamma\geq 1$ our functional is weakly strictly decreasing~\cite{MR3506069}, that is if $\lambda_k(\Omega)<\lambda_k(\widetilde{\Omega})$ for all $k\geq 1$ then 
  \begin{equation}
    \Tr(-\Delta_\Omega-\specParam)_\limminus^\gamma > \Tr(-\Delta_{\widetilde{\Omega}}-\specParam)_\limminus^\gamma.
  \end{equation}
  Moreover, by the Li--Yau inequality~\eqref{eq:LiYau}, our functional is for any fixed $\specParam$ a finite sum of Lipschitz functions and hence Lipschitz. Thus the last part of the lemma is a direct consequence of Theorem~3.4 in~\cite{MR2044433}. 

  If $\specParam\leq \inf\{\lambda_1(\Omega): \Omega\in \A, |\Omega|=1\}$ then the supremum in~\eqref{eq:MaxProb_Existence_thm} is zero and hence any domain in $\A$ is a maximizer. If this is not the case we let $\{\Omega_k\}_{k\geq1}\subset \A$, with $|\Omega_k|=1$, be a maximizing sequence for~\eqref{eq:MaxProb_Existence_thm}. Without loss of generality we assume that $\Tr(-\Delta_{\Omega_k}-\specParam)^\gamma_\limminus>0$. In particular, we must have that $\lambda_1(\Omega_k)< \specParam$ for all $k$. Hence the inequality $\lambda_1(\Omega)\geq \frac{\pi^2}{4r(\Omega)^2}$, for $\Omega \subset \K^n$, due to Hersch in $\R^2$ and Protter in $\R^n$~\cite{Hersch,MR589137} implies that $r(\Omega_k)> \frac{\pi}{2\sqrt{\specParam}}$. Since $|\Omega_k|=1$ for each $k$ we by~\eqref{eq:diameter_bound} obtain an upper bound for $D(\Omega_k)$ which is independent of~$k$. 

  As our functional is invariant under translation we may translate each $\Omega_k$ so that it has barycentre at the origin and obtain a new maximizing sequence which is uniformly bounded. By the Blaschke selection theorem~\cite[Theorem~1.8.7]{Schneider1} we can extract a subsequence which converges in $\K^n$, and hence in $\A$. Abusing notation denote this subsequence by $\{\Omega_k\}_{k\geq 1}$ and let $\Omega_\infty$ denote its limit. Since the eigenvalues of the Dirichlet Laplacian are lower-semi continuous with respect to the topology on $\K^n$~\cite{MR2512810} we find that $\Omega_\infty$ realizes the supremum in~\eqref{eq:MaxProb_Existence_thm}.
\end{proof}


\section{Convergence of maximizers}\label{sec:Convergence}

In this section we prove that for any sequence $\{\specParam_j\}_{j\geq 1}$ tending to infinity the corresponding sequence of maximizers $\Omega_{\specParam_j, \gamma}(\A)$ 
has a convergent subsequence. Moreover, if $\A$ satisfies an additional assumption we characterize the possible limit points of such subsequences. Our main objective is to prove the following proposition:
\begin{proposition}\label{prop:Asymptotic_Shape_Thm}
  Let $\A$ be a closed subset of $\K^n$. Fix $\gamma \geq 1$ and let\/ $\Omega_{\specParam, \gamma}(\A)$ denote any extremal domain for the shape optimization problem
  \begin{equation}\label{eq:prob_asymptotic_shape}
     \sup\{\Tr(-\Delta_\Omega-\specParam)^\gamma_\limminus : \Omega\in \A, |\Omega|=1\}.
   \end{equation}
  Then the following statements hold:
  \begin{enumerate}[label=\textup{(\roman*)}]
  \item\label{itm:Prop_I} For any sequence $\{\specParam_j\}_{j\geq 1}\!\!\uparrow\!\infty$ the corresponding sequence $\{\Omega_{\specParam_j, \gamma}(\A)\}_{j\geq 1}$ has a subsequence which, up to translation, converges in~$\A$. Moreover, $\Omega_\infty$ the limit of such a subsequence has unit measure.

  \item\label{itm:Prop_II} Under the additional assumption that
  \begin{equation}
    \Tr(-\Delta_\Omega-\specParam)^\gamma_\limminus = \LT{\gamma}n |\Omega| \specParam^{\gamma+n/2}-\frac{1}{4}\LT{\gamma}{n-1}|\partial\Omega|\specParam^{\gamma+(n-1)/2}+o(\specParam^{\gamma+(n-1)/2}),
  \end{equation}
  as $\specParam \to \infty$, uniformly for $\Omega$ in compact subsets of $\A$, then the limit $\Omega_\infty$ also minimizes the perimeter in $\A$:
  \begin{equation}\label{eq:prop_perimeter_minimizer}
    |\partial\Omega_\infty|=\inf\{ |\partial\Omega|: \Omega \in \A, |\Omega|=1\}.
  \end{equation}
\end{enumerate}
\end{proposition}

\begin{remark} 
    As a consequence of the results in~\cite{FrankLarson} we know that the assumption in the second part of the theorem is redundant, the expansion holds uniformly on any compact subset of $\K^n$.
    As a consequence the conclusions of Proposition~\ref{prop:Asymptotic_Shape_Thm} remain true without it and hence extends Theorems~\ref{thm:Polytopes} and~\ref{thm:UnifC1} to \emph{any} admissible class of convex domains $\A$ (see~\cite{FrankLarson}). 
\end{remark}

With Proposition~\ref{prop:Asymptotic_Shape_Thm} in hand it is straightforward to prove our main theorems.
\begin{proof}[Proof of Theorems~\ref{thm:Polytopes} and~\ref{thm:UnifC1}]
  By Lemmas~\ref{lem:Asymptotic_Komega} and~\ref{lem:Asymptotic_Polytopes} the classes $\K^n_\omega$ and $\mathcal{P}_m$ satisfy the assumptions of~\ref{itm:Prop_I}-\ref{itm:Prop_II} of Proposition~\ref{prop:Asymptotic_Shape_Thm}, and thus the theorems follow as special cases thereof.
\end{proof}

\begin{proof}[Proof of Proposition~\ref{prop:Asymptotic_Shape_Thm}]
The proof follows closely the strategy of Antunes and Freitas~\cite{MR3001382} (see also~\cite{vdBBG,vdBergGittins,MR3098942,GittinsLarson} for applications in very similar settings): Using the bound of Theorem~\ref{thm:Improved_Berezin_Convex} one readily obtains that the sequence of maximizers have uniformly bounded perimeters. Using the inequalities of Section~\ref{sec:Prel_convex} we can conclude that the sequence is uniformly bounded, and thus extract a convergent subsequence. The final ingredient is to use the uniform asymptotic expansions in~\ref{itm:Prop_II} to identify the limiting domains as minimizers of the perimeter.

Fix $\A$ and $\gamma \geq 1$. For notational convenience we will for a maximizer of~\eqref{eq:prob_asymptotic_shape} write simply $\Omega_\specParam$ instead of $\Omega_{\specParam, \gamma}(\A)$. Without loss of generality we throughout the proof assume that the barycentre of each maximizer is the origin. The idea used to prove the existence of a convergent subsequence of $\Omega_{\specParam, \gamma}(\A)$ is to
use the maximality of $\Tr(-\Delta_{\Omega_{\specParam}}-\specParam)_\limminus^\gamma$ and compare it with the corresponding Riesz mean for some fixed domain $\Omega_0\in\A$ with $|\Omega_0|=1$.

Assume that  $\specParam>\inf\{\lambda_1(\Omega): \Omega\in \A, |\Omega|=1\}$. Then, by the maximality of~$\Omega_{\specParam}$,
\begin{equation}
  0<\Tr(-\Delta_{\Omega_0}-\specParam)^\gamma_\limminus \leq \Tr(-\Delta_{\Omega_{\specParam}}-\specParam)^\gamma_\limminus.
\end{equation}
Using Theorem~\ref{thm:Improved_Berezin_Convex} and Lemma~\ref{lem:one_sided_two_term_asymptotics} this inequality implies that
\begin{equation}\label{eq:comparison_boundedness}
\begin{aligned}
  \LT{\gamma}{n}\specParam^{\gamma+n/2}-\frac{\LT{\gamma}{n-1}}{4}|\partial \Omega_0| &\specParam^{\gamma+(n-1)/2} +o(\specParam^{\gamma+(n-1)/2})\\
  &\leq 
  \LT{\gamma}{n}\specParam^{\gamma+n/2} - c(\gamma, n)\LT{\gamma}{n-1}|\partial \Omega_{\specParam}|\specParam^{\gamma+ (n-1)/2}.
\end{aligned}
\end{equation}
Rearranging~\eqref{eq:comparison_boundedness} yields
\begin{equation}
  |\partial \Omega_\specParam| \leq \frac{|\partial \Omega_0|}{4c(\gamma, n)} + o(1), 
\end{equation}
as $\specParam \to \infty$, and thus the perimeter of the maximizers remains uniformly bounded in $\specParam$. By~\eqref{eq:twosided_inradius_bound} and~\eqref{eq:diameter_bound} we conclude that $\Omega_\specParam$ remains uniformly bounded with respect to $\specParam$. Thus we can for any sequence $\{\specParam_j\}_{j\geq 1}$ tending to infinity extract a subsequence of $\{\Omega_{\specParam_j}\}_{j\geq 1}$ which converges to a domain $\Omega_\infty\in \A$. Since $|\Omega|$ and $|\partial\Omega|$ are continuous with respect to the topology of $\K^n$ we find that $|\Omega_\infty|=1$ and $|\partial\Omega_\infty|\leq \frac{|\partial \Omega_0|}{4c(\gamma, n)}$, this completes the proof of~\ref{itm:Prop_I}.

With a slight abuse of notation we let $\{\specParam_j\}_{j\geq 1}$ denote the subsequence along which $\{\Omega_{\specParam_j}\}_{j\geq 1}$ converges to $\Omega_\infty$. For each $j\geq 1$ we have, by the maximality of $\Omega_{\specParam_j}$, that
\begin{equation}
  \frac{\Tr(-\Delta_{\Omega_0}-\specParam_j)^\gamma_\limminus-\LT{\gamma}{n}\specParam_j^{\gamma+n/2}}{\specParam_j^{\gamma+(n-1)/2}} \leq \frac{ \Tr(-\Delta_{\Omega_{\specParam_j}}-\specParam_j)^\gamma_\limminus-\LT{\gamma}{n}\specParam_j^{\gamma+n/2}}{\specParam_j^{\gamma+(n-1)/2}}.
\end{equation}
Assume now that $\A$ satisfies the assumption in~\ref{itm:Prop_II}.
Using that our sequence of maximizers $\{\Omega_{\specParam_j}\}_{j\geq 1}$ is bounded, and hence contained in a compact subset of $\A$, to uniformly control the error terms, one finds that
\begin{equation}
  |\partial\Omega_{\specParam_j}| \leq |\partial \Omega_0| + o(1), 
\end{equation}
as $j\to \infty$. Since the sequence $\Omega_{\specParam_j}$ converges to $\Omega_{\infty}$ and the measure of the perimeter is continuous in the topology of $\K^n$, we obtain that $|\partial \Omega_{\infty}|\leq |\partial \Omega_0|$. Choosing $\Omega_0$ to realize the infimum in~\eqref{eq:prop_perimeter_minimizer} concludes the proof.
\end{proof}

\begin{remark}
  We note that in the proof of~\ref{itm:Prop_I} we do not require the full statement of Lemma~\ref{lem:one_sided_two_term_asymptotics} only that there exists \emph{one} domain $\Omega_0\in \A$ with $|\Omega_0|=1$ for which the second term of the asymptotic expansion of the Riesz mean is of the correct order $\sim \specParam^{\gamma+(n-1)/2}$.
\end{remark}


\section{Sums of eigenvalues}\label{sec:sums}

In this section we prove that our techniques allow us also to study the behaviour of convex domains realizing the infimum
\begin{equation}\label{eq:shape_prob_sums}
   \inf\Bigl\{\frac{1}{m}\sum_{k=1}^m \lambda_k(\Omega): \Omega \in \A, |\Omega|=1\Bigr\}
 \end{equation} 
in the limit $m\to \infty$. This problem, but without the convexity assumptions, was recently studied by Freitas~\cite{FreitasAverages}.

\begin{theorem}\label{thm:sums}
  Let $\A$ be a closed subset of $\K^n$ satisfying the assumption in~\ref{itm:Prop_II} of Proposition~\ref{prop:Asymptotic_Shape_Thm}. Let $\Omega_m(\A)$ denote any extremal domain for the shape optimization problem~\eqref{eq:shape_prob_sums}. Then the sequence $\{\Omega_m\}_{m\geq 1}$ has a subsequence which, up to translations, converges in~$\A$. Moreover, $\Omega_\infty$ the limit of such a subsequence has unit measure and minimizes the perimeter in $\A$:
  \begin{equation}
     |\partial\Omega_\infty|=\inf\{|\partial\Omega|: \Omega\in \A, |\Omega|=1\}.
  \end{equation} 
\end{theorem}
\begin{remark}
  Again we note that the extra assumption on $\A$ can be dropped in light of the results in~\cite{FrankLarson}.
\end{remark}

The proof of Theorem~\ref{thm:sums} is based on our tools developed for Riesz means and the close connection between sums of eigenvalues and Riesz means of order $\gamma=1$. In particular, via the Legendre transform the asymptotic expansion for $\Tr(-\Delta_\Omega-\specParam)_\limminus$ implies a similar two-term expansion for the sum (see, for instance,~\cite[Appendix A]{MR3466545}): For $\Omega\subset \R^n$ such that~\eqref{eq:Two_Term_Weyl} holds
\begin{equation}
  \frac{1}{m}\sum_{k=1}^m \lambda_k(\Omega) = A_n \Bigl(\frac{m}{|\Omega|}\Bigr)^{2/n}+ B_n\frac{|\partial\Omega|}{|\Omega|}\Bigl(\frac{m}{|\Omega|}\Bigr)^{1/n}+ o(m^{1/n}),
\end{equation}
as $m\to \infty$. The constants $A_n, B_n$ are explicitly given by
\begin{equation}
  A_n= \frac{4\pi n \Gamma(\frac{n}{2}+1)^{2/n}}{n+2}, \qquad B_n=\frac{2\pi \Gamma(\frac{n}{2}+1)^{1+1/n}}{(n+1)\Gamma(\frac{n+1}{2})}.
\end{equation}
It should also be noted that the Legendre transform switches the direction of inequalities. In particular, the lower bound for the Riesz mean asymptotics provided by Lemma~\ref{lem:one_sided_two_term_asymptotics} turns into a corresponding upper bound for the asymptotics of the sum.

If we can prove a bound similar to Theorem~\ref{thm:Improved_Berezin_Convex} in the setting of eigenvalue sums then it is straightforward to follow the strategy in the proofs of Lemma~\ref{lem:Existence_of_maximizer} and Proposition~\ref{prop:Asymptotic_Shape_Thm} to prove first the existence and uniform boundedness of the minimizers.

\begin{corollary}[Improved Li--Yau inequality]\label{cor:Improved_LiYau_Convex}
Let\/ $\Omega\in \K^n$. There exists a positive constant ${c(n)}$ such that, for all $m\geq 1$,
\begin{align}
    \frac{1}{m}\sum_{k=1}^m  \lambda_k(\Omega) \geq A_n\Bigl(\frac{m}{|\Omega|}\Bigr)^{2/n}+ c(n)B_n\frac{|\partial\Omega|}{|\Omega|}\Bigl(\frac{m}{|\Omega|}\Bigr)^{1/n}.
\end{align}
\end{corollary}

\begin{proof}[Proof of Corollary~\ref{cor:Improved_LiYau_Convex}]
The proof is based on the fact that 
\begin{equation}\label{eq: Legendre transform}
  \sup_{\Lambda \geq 0}\Bigl(m\Lambda - \sum_{k:\lambda_k(\Omega)<\Lambda}(\Lambda-\lambda_k(\Omega))\Bigr)= \sum_{k=1}^m \lambda_k(\Omega).
\end{equation}
This identity can be deduced by studying the sign of the derivative of the expression in the parenthesis with respect to $\specParam$ on intervals where $N_\Omega(\specParam)$ is constant. 
To find a bound for the right-hand side of~\eqref{eq: Legendre transform} we plug in a suitable choice of $\Lambda \geq 0$ in the left-hand side and use Theorem~\ref{thm:Improved_Berezin_Convex} to bound the Riesz mean.

The natural choice to make is $\Lambda=\Lambda_0:= \frac{4^{1/n}}{(n+2)^{2/n}(\LT{1}n)^{2/n}}\bigl(\frac{m}{|\Omega|}\bigr)^{2/n}$, however we need to distinguish whether this quantity is less or greater than $\frac{\pi^2}{4r(\Omega)^2}$. (This case distinction was missed in the published version of the paper, resulting in an erroneous proof). If $\Lambda_0 \geq \frac{\pi^2}{4r(\Omega)^2}$ Theorem~\ref{thm:Improved_Berezin_Convex} implies that
\begin{align*}
  \sum_{k=1}^m \lambda_k(\Omega) 
  &\geq 
  m \Lambda_0 - \LT{1}n |\Omega|\Lambda_0^{1+n/2}+ c\LT{1}{n-1}|\partial\Omega|\Lambda_0^{(n+1)/2}\\
  &=
  m A_n \Bigl(\frac{m}{|\Omega|}\Bigr)^{2/n}+ mc' B_n \frac{|\partial\Omega|}{|\Omega|}\Bigl(\frac{m}{|\Omega|}\Bigr)^{1/n}.
\end{align*}

If $\Lambda_0\leq \frac{\pi^2}{4r(\Omega)^2}$ choose $\Lambda = \frac{\pi^2}{4r(\Omega)^2}$. The results of~\cite{Hersch,MR589137} (or Theorem~\ref{thm:Improved_Berezin_Convex}) implies that $\frac{1}{m}\sum_{k=1}^m \lambda_k(\Omega)\geq \frac{\pi^2}{4r(\Omega)^2}$. What remains is to show that if $\Lambda_0 \leq \frac{\pi^2}{4r(\Omega)^2}$, then
\begin{equation}\label{eq: goal}
   A_n \Bigl(\frac{m}{|\Omega|}\Bigr)^{2/n}+ c B_n \frac{|\partial\Omega|}{|\Omega|}\Bigl(\frac{m}{|\Omega|}\Bigr)^{1/n} \leq \frac{\pi^2}{4r(\Omega)^2},
\end{equation}
provided $c=c(n)>0$ is small enough. Let $a= \frac{4^{1/n}}{(n+2)^{2/n}(\LT{1}n)^{2/n}}$ so that $A_n = a-\LT{a}{n}a^{1+n/2}$. Therefore, by~\eqref{eq:twosided_inradius_bound} and $\Lambda_0=a\bigl(\frac{m}{|\Omega|}\bigr)^{2/n}\leq \frac{\pi^2}{4r(\Omega)^2}$,
\begin{equation}
  A_n\Bigl(\frac{m}{|\Omega|}\Bigr)^{2/n}+ cB_n \frac{|\partial\Omega|}{|\Omega|}\Bigl(\frac{m}{|\Omega|}\Bigr)^{1/n}\leq \Bigl[(1-\LT{1}na^{n/2})+cB_n \frac{2n}{\pi a^{1/2}}\Bigr] \frac{\pi^2}{4r(\Omega)^2},
\end{equation}
from which~\eqref{eq: goal} follows for $c$ sufficiently small.
\end{proof}

\begin{proof}[Proof of Theorem~\ref{thm:sums}]
  The claim follows by mimicking the proofs of Lemmas~\ref{lem:Existence_of_maximizer} and Proposition~\ref{prop:Asymptotic_Shape_Thm}. The use of the asymptotic bound of Lemma~\ref{lem:one_sided_two_term_asymptotics} should be replaced by its corresponding Legendre transform, and the use of Theorem~\ref{thm:Improved_Berezin_Convex} by Corollary~\ref{cor:Improved_LiYau_Convex}.
\end{proof}


\section{Uniform two-term asymptotics}\label{sec:Asymptotics}

In this section we use the methods of Frank and Geisinger~\cite{MR2885166,MR2994205} to prove Lemmas~\ref{lem:one_sided_two_term_asymptotics} and~\ref{lem:Asymptotic_Polytopes}. The proof of Lemma~\ref{lem:one_sided_two_term_asymptotics} will complete the proof of Proposition~\ref{prop:Asymptotic_Shape_Thm}, which in combination with Lemma~\ref{lem:Asymptotic_Komega} and Lemma~\ref{lem:Asymptotic_Polytopes} proves Theorem~\ref{thm:UnifC1} and Theorem~\ref{thm:Polytopes}, respectively.

To match the notation in~\cite{MR2885166,MR2994205} we here consider the asymptotics of $\Tr(-h^2\Delta_\Omega-1)_\limminus^\gamma$ as $h\to 0^\limplus$. By a simple calculation~\eqref{eq:Two_Term_Weyl} is equivalent to
\begin{equation}
  \Tr(-h^2 \Delta_\Omega-1)_\limminus^\gamma = \LT{\gamma}{n}|\Omega|h^{-n}- \frac{\LT{\gamma}{n-1}}{4}|\partial \Omega|h^{-n+1}+o(h^{-n+1}), \quad \mbox{as }h\to 0^\limplus,
\end{equation}
and~\eqref{eq:lem_onesided_asymptotics} to the corresponding inequality.

In~\cite{MR2885166, MR2994205} the authors consider only the case $\gamma=1$ but it can be lifted to larger $\gamma$ using the Aizenman--Lieb identity~\cite{AizenmanLieb}: If $\gamma_1\geq 0$ and $\gamma_2>\gamma_1$ then
  \begin{equation}
    \Tr(-\Delta_\Omega-\specParam)^{\gamma_2}_\limminus = B(1+\gamma_1, \gamma_2-\gamma_1)^{-1} \int_0^\infty \tau^{-1+\gamma_2-\gamma_1} \Tr(-\Delta_\Omega-(\specParam-\tau))^{\gamma_1}_\limminus\, d\tau, 
  \end{equation}
  where $B$ denotes the Euler Beta function. It thus suffices to prove Lemmas~\ref{lem:one_sided_two_term_asymptotics} and~\ref{lem:Asymptotic_Polytopes} in the case $\gamma=1$.

The proof relies on localizing the operator into balls whose sizes vary depending on the distance to the complement of $\Omega$. The asymptotic contributions from each of these localizations is then analysed separately. 

Using Theorem~22 in~\cite{SolovejSpitzer} the localization is constructed by introducing a length-scale $l(u)$ and functions $\phi_u \in C^\infty_0(\R^n; \R)$ with support in $B_{l(u)}(u)=\{x\in \R^n : |x-u|< l(u)\}$, satisfying that
\begin{equation}\label{eq:Assumptions_phi}
  \|\phi_u\|_{\infty}\leq c, \quad \|\grad \phi_u\|_{\infty}\leq c\pps l(u)^{-1}
\end{equation}
and for any $x \in \R^n$
\begin{equation}\label{eq:Integral_partition_unity}
  \int_{\R^n} \phi_u^2(x)l(u)^{-n}\, du = 1.
\end{equation}
Here and in what follows $c$ will denote a positive constant which may change from line to line, but which depends only on the dimension and the choice of $l(u), \phi_u$.
Following~\cite{MR2885166,MR2994205} we set
\begin{equation}\label{eq:local_length_scale}
  l(u):=\frac{1}{2}\bigl(1+(\dist(u, \Omega^c)^2+l_0^2)^{-1/2}\bigr)^{-1},
\end{equation}
where $l_0\in (0, 1)$ is a parameter depending only on $h$ which will tend to zero as $h\to 0^\limplus$.

We will use the following results from~\cite{MR2885166, MR2994205}:

\begin{lemma}[{\cite[Proposition~1.1]{MR2885166}}]\label{lem:Localization}
 For\/ $0<l_0< 1$ and $0<h<M l_0$ we have that
\begin{equation}
  \biggl|\Tr(-h^2\Delta_\Omega-1)_\limminus - \int_{\R^n}\Tr(\phi_u(-h^2\Delta_\Omega-1)\phi_u)_\limminus l(u)^{-n}\, du \biggr| \leq 
  c \int_{\Omega^*}l(u)^{-2}\,du\,h^{-n+2},
\end{equation}
where $\Omega^*:=\{u\in \R^n: \supp\phi_u\cap \Omega \neq \emptyset\}$ and the constant $c$ depends only on $M$ and those in~\eqref{eq:Assumptions_phi}.
\end{lemma}
\begin{remark}
  In Proposition~1.1 in~\cite{MR2885166} the integral on the right-hand side is in the final step of the proof bounded in terms of $l_0^{-1}$. As we here wish to keep track of how the remainder depends on $\Omega$ we choose to keep it in integral form. 

 Moreover, in~\cite{SolovejSpitzer} the function $l$ is assumed to be $C^1$-regular, however, a Lipschitz assumption is sufficient (see~\cite{FrankLarson}).
\end{remark}
\begin{remark}
  We also note that $\Omega\subsetneq \Omega^* \subseteq \Omega^t$, for $t= \frac{l_0}{2+2l_0}$. If the functions $\phi_u$ are chosen so that $\supp\phi_u=B_{l(u)}(u)$ then equality holds in the second inclusion.
\end{remark}

\begin{lemma}[{\cite[Proposition~1.1]{MR2885166} and \cite[Proposition~2.3]{MR2994205}}]\label{lem:BoundaryEstimate}
  Let $\phi \in C^\infty(\R^n)$ be supported in a ball of radius $l>0$ and satisfy
  \begin{equation}\label{eq:grad_bound_localgraph}
    \|\grad \phi\|_\infty \leq cl^{-1}.
  \end{equation}
  Assume that the intersection $\partial \Omega \cap \supp \phi$ is $C^1$ with modulus of continuity $\omega\colon [0, L)\to \R$, with $L\geq 2l$, in the sense of \ref{itm:def_Komega}. 

  Then if $l$ is so small that $\omega(l)\leq c_n$, where $c_n$ depends only on the dimension, it holds for $0<h\leq l$ that
  \begin{equation}
    \biggl| \Tr(\phi(-h^2\Delta_\Omega-1)\phi)_\limminus - \LT{1}{n}\int_\Omega \phi^2(x)\, dx\, h^{-n}  + \frac{1}{4}\LT{1}{n-1}  \int_{\partial \Omega} \phi^2(x) \, d\sigma(x)\, h^{-n+1}\biggr| \leq r(l, h),
  \end{equation}
  where $d\sigma$ denotes the $(n-1)$-dimensional Lebesgue measure on\/ $\partial\Omega$ and the remainder satisfies
  \begin{equation}
    r(l, h) \leq c \Bigl(\frac{l^{n-2}}{h^{n-2}}+\frac{\omega(l)^2 l^{n-1}}{h^{n-1}}+\frac{\omega(l)l^n}{h^n}\Bigr),
  \end{equation}
  where the constant $c$ depends only on that in~\eqref{eq:grad_bound_localgraph}.
\end{lemma}
\begin{remark}
    Here we shall only make use of Lemma~\ref{lem:BoundaryEstimate} when the boundary of $\Omega$ is either $C^{1,1}$-regular or when $\Omega\in\mathcal{P}_m$ and the boundary is locally a hyperplane, in the latter case we can take $\omega\equiv 0$. In~\cite{MR2885166,MR2994205} it is stated that the smallness assumption on $l$ may depend on $\Omega$, this is however not necessary the relevant local geometry is encoded by $\omega$. Inspection of the proofs in~\cite{MR2885166,MR2994205} yields that one can can take $c_n= \tfrac{1}{4(n-1)}.$
\end{remark}

We shall also need the following lemma which can be viewed as a local version of~\eqref{eq:BerezinIneq}.
\begin{lemma}[{\cite[Lemma~2.1]{MR2885166}}]\label{lem:LocalBerezin}
  For any $\phi \in C^\infty_0 (\R^n)$ and $h>0$ we have that
  \begin{equation}
    \Tr(\phi(-h^2\Delta_\Omega-1)\phi)_\limminus \leq \LT{1}{n} \int_{\Omega}\phi^2(x)\, dx\, h^{-n}.
\end{equation}
\end{lemma}

To prove Lemma~\ref{lem:one_sided_two_term_asymptotics} we will need a more refined version of this inequality when the support of $\phi$ is disjoint from the boundary of $\Omega$.

\begin{lemma}[{\cite[Proposition~1.2]{MR2885166}}]\label{lem:localBerezin_bulk}
  Let $\phi\in C^\infty_0(\Omega)$ be supported in a ball of radius $l>0$ and satisfy
  \begin{equation}\label{eq:grad_bound_bulk}
    \|\grad\phi\|_\infty \leq c l^{-1}.
  \end{equation}
  Then, for all $h>0$,
  \begin{equation}
    \biggl| \Tr(\phi (-h^2\Delta_\Omega-1)\phi)_\limminus - \LT{1}{n}\int_\Omega \phi^2(x)\,dx \, h^{-n}\biggr| \leq c\pps l^{n-2}h^{-n+2},
  \end{equation}
  where the constant $c$ depends only on that in~\eqref{eq:grad_bound_bulk}.
\end{lemma}

To control the error terms coming from the applications of the local bounds above we shall need the following inequalities which appear in~\cite{MR2885166} (or with explicitly stated geometric dependence in~\cite{FrankLarson}): For $\Omega \in \K^n$ and $\alpha \in \R$ it holds that
\begin{align}
   \int_{\Omega_*}l(u)^{-2}\,du &\leq c(|\Omega|+|\partial\Omega|)l_0^{-1},\label{eq:bulk_integralbound}\\
   \int_{\Omega^*\setminus \Omega_*}l(u)^{\alpha}\,du &\leq c |\partial\Omega|l_0^{1+\alpha},\label{eq:bdry_integralbound}
 \end{align}
 where $\Omega^*$ is defined as in Lemma~\ref{lem:Localization}, $\Omega^*=\{u\in \R^n: \supp\phi_u \cap \Omega \neq \emptyset\}$, and similarly $\Omega_*:=\{u\in \Omega : \supp\phi_u\subset\Omega\}$. As noted above $\Omega^*$ is essentially an outer parallel set of~$\Omega$. Similarly $\Omega_*$ is essentially an inner parallel set. In particular we note the inclusions $\Omega_*\subset \Omega\subset \Omega^*$.

Using the above we are ready to prove Lemmas~\ref{lem:one_sided_two_term_asymptotics} and~\ref{lem:Asymptotic_Polytopes}.

\begin{proof}[{Proof of Lemma~\ref{lem:one_sided_two_term_asymptotics}}]
  The proof is based on constructing a nested family of regular convex domains $\Omega(\eps)\in \K^n$, for $\eps>0$, such that $\Omega(0) = \Omega$ and $\Omega(\eps) \subset \Omega(\eps')$ if $\eps>\eps'$.

  Define, in the notation introduced in Section~\ref{sec:Prel_convex}, $\Omega(\eps) := (\Omega_\eps)^\eps = (\Omega\sim B_\eps)+B_\eps$, that is the outer parallel set of the inner parallel set of $\Omega$ at distance $\eps>0$. For $0\leq\eps<r(\Omega)$ it is clear from the construction that $\Omega(\eps)$ are non-empty and nested as described above. We also see that $\Omega(\eps)$ satisfies an $\eps$-inner ball condition, and hence its boundary is $C^{1,1}$-regular (see, for instance,~\cite{CaffarelliCabre}). Furthermore, it holds that $D(\Omega(\eps))\leq D(\Omega)$ and $r(\Omega(\eps))=r(\Omega)$.

  By~\eqref{eq:Steinerformula} and the properties of mixed volumes listed in Section~\ref{sec:Prel_convex} we have
  \begin{align}
    \bigl||\Omega^\eps| - |\Omega|- \eps |\partial\Omega|\bigr| = \sum_{j = 2}^n \binom{n}{j} \eps^j W(\underbrace{\Omega, \ldots, \Omega}_{n-j}, \underbrace{B_1, \ldots, B_1}_{j})
  \end{align}
  and 
  \begin{align}
    \sum_{j = 2}^n \binom{n}{j} \eps^j W(\underbrace{\Omega, \ldots, \Omega}_{n-j}, \underbrace{B_1, \ldots, B_1}_{j}) 
    &= 
    \sum_{j = 2}^n \binom{n}{j} \frac{\eps^j}{D(\Omega)^j} W(\underbrace{\Omega, \ldots, \Omega}_{n-j}, \underbrace{D(\Omega)B_1, \ldots, D(\Omega)B_1}_{j})\\[-10pt]
    &\leq
    \sum_{j = 2}^n \binom{n}{j} \frac{\eps^j}{D(\Omega)^j} D(\Omega)^n |B_1| \leq c D(\Omega)^{n-2} \eps^2.
  \end{align}
  Similarly
  \begin{align}
    \bigl||\partial\Omega^\eps|-|\partial\Omega|\bigr|
    = 
    n \sum_{j=1}^{n-1} \binom{n-1}{j} \eps^j W(\underbrace{\Omega, \ldots, \Omega}_{n-j-1}, \underbrace{B_1, \ldots, B_1}_{j+1})
    \leq
     c D(\Omega)^{n-2}\eps.
  \end{align}
  Hence we can conclude that
  \begin{equation}\label{eq:outer_volper_bound}
    |\Omega^\eps| = |\Omega|+ \eps |\partial \Omega| + O(\eps^2) \quad \mbox{and} \quad |\partial\Omega^\eps| = |\partial \Omega| + O(\eps),
  \end{equation}
  where both error terms are uniform on compact subsets of $\K^n$.
  
Moreover, by~\eqref{eq:bound_per_innerparallel} and the corresponding upper bound
 \begin{equation}\label{eq:inner_perbound}
    |\partial\Omega_\eps| = |\partial \Omega | + O(\eps),
  \end{equation}
  where the implicit constant can be bounded from above by a constant times $|\partial\Omega|/r(\Omega)$.

  Combining $\Omega(\eps)\subseteq \Omega$ with~\eqref{eq:outer_volper_bound},~\eqref{eq:inner_perbound} and the inequality $ |\Omega_\eps|\geq |\Omega|-\eps |\partial \Omega|$ yields that
  \begin{equation}\label{eq:vol_per_control}
     |\Omega(\eps)| = |\Omega| + O(\eps^2) \quad \mbox{and} \quad |\partial \Omega(\eps)| = |\partial \Omega| + O(\eps),
  \end{equation}
where again both the error terms are uniform on compact subsets of $\K^n$.

  By the monotonicity of Dirichlet eigenvalues under domain inclusion
  \begin{align}\label{eq:monotonicity_innerapprox}
    \Tr(-h^2\Delta_\Omega-1)_\limminus \geq \Tr(-h^2\Delta_{\Omega(\eps)}-1)_\limminus.
  \end{align}
  The idea is now to apply the methods of~\cite{MR2885166} to each $\Omega(\eps)$, keeping track of how the error terms depend on $\eps$, and in the final step choose $\eps$ appropriately depending on $h$.

  We first observe that if $x\in \partial\Omega(\eps)$ and $\delta\leq \eps/2$ then the set $\partial\Omega(\eps)\cap B_\delta(x)$ is (in the sense above) locally a graph of a $C^{1,1}$-function $f$ satisfying
  \begin{equation}\label{eq:modcontinuity_eps}
     |\grad f(x')-\grad f(y')| \leq \frac{c}{\eps} |x'-y'|,
   \end{equation} 
   where $c$ is a dimensional constant. Indeed, by convexity and the fact that $\Omega$ satisfies a uniform $\eps$-inner ball condition it follows that $f$ is $C^{1,1}$-regular from Propositions~1.1 and~1.2 in~\cite{CaffarelliCabre}. That the constant in the $C^{1,1}$-estimate~\eqref{eq:modcontinuity_eps} behaves like $\eps^{-1}$ is a consequence of scaling: If $f(x)$ can be touched from above and below at each point by a ball of radius $\eps$ then $g(x):=f(\eps x)/\eps$ can at each point be touched from above and below  by a ball of radius~$1$.

  Let $l(u)$ be defined as in~\eqref{eq:local_length_scale} with respect to the set $\Omega(\eps)$ with $l_0\in (0, 1)$ to be chosen later, and $\phi_u$ be the corresponding family of functions (we emphasize that this definition depends on $\eps$ even though this is not reflected in our notation). 

  Consider the set
  \begin{align}
    \Omega^*(\eps) &:= \{u \in \R^n : \supp \phi_u \cap \Omega(\eps) \neq \emptyset\},
  \end{align}
  we note again that $\Omega^*(\eps)$ contains points in the complement of $\Omega(\eps)$. This is precisely the set of $u\in \R^n$ where $\Tr(\phi_u(-h^2\Delta_{\Omega(\eps)}-1)\phi_u)_\limminus$ is non-zero. We split $\Omega^*(\eps)$ into the sets $\Omega_*(\eps):=\{u\in \Omega^*(\eps): \supp \phi_u \subset \Omega(\eps)\}$ and $\Omega_b(\eps):=\Omega^*(\eps)\setminus \Omega_*(\eps)$. 
  The set $\Omega_*(\eps)$ is precisely the set of $u \in \Omega^*(\eps)$ such that $\supp \phi_u \cap \partial\Omega(\eps) = \emptyset$, and $\Omega_b(\eps)$ is the set where the same intersection is non-empty. 

  Let $t^*$ solve the equation $t = \frac{1}{2}\bigl(1+(t^2+l_0^2)^{-1/2}\bigr)^{-1}=l(u)\big|_{\dist(u, \Omega^c)=t}$. By observing that $l(u)\big|_{u\in \Omega^c}= \frac{l_0}{2l_0+2}$ and $0\leq \frac{d}{dt}\bigl(\frac{1}{2}\bigl(1+(t^2+l_0^2)^{-1/2}\bigr)^{-1}\bigr)\leq \frac{1}{2}$ it is clear that $t^*$ is unique, and moreover that $t^* \leq l_0/{\sqrt{3}}$ since
  \begin{equation}
    \frac{1}{2}\bigl(1+\bigl((l_0/{\sqrt{3}})^2+l_0^2\bigr)^{-1/2}\bigr)^{-1} = \frac{l_0}{\sqrt{3}+2l_0} \leq \frac{l_0}{\sqrt{3}}.
  \end{equation}
  By the remarks above $l(u)\geq l_0/4$ for all $u\in \R^n$, and moreover since $\Omega_b(\eps)$ is precisely the set where $l(u)\geq \dist(u, \partial\Omega)$ we find that if $u\in \Omega_b(\eps)$ then $l(u)\leq l_0/{\sqrt{3}}$.

  By Lemma~\ref{lem:Localization} and~\eqref{eq:bulk_integralbound} we have for $0<h<M l_0$ and $\eps\in [0, r(\Omega))$ that
\begin{align}\label{eq:int_split}
    \Tr(-h^2\Delta_{\Omega(\eps)}-1)_\limminus 
  &\geq
    \int_{\Omega_*(\eps)} \Tr(\phi_u(-h^2\Delta_{\Omega(\eps)}-1)\phi_u)_\limminus l(u)^{-n}\, du \\
    & + \int_{\Omega_b(\eps)} \Tr(\phi_u(-h^2\Delta_{\Omega(\eps)}-1)\phi_u)_\limminus l(u)^{-n}\, du 
    - c(|\Omega|+|\partial\Omega|)l_0^{-1}h^{-n+2},\\[-30pt]
\end{align}
where the constant in the error term can be chosen independent of $\eps$ due to~\eqref{eq:vol_per_control}.

If $u\in \Omega_b(\eps)$ then $\dist(u, \partial\Omega)\leq l(u)\leq l_0/{\sqrt{3}}$ and thus $B_{l(u)}(u)\subset B_{2l_0/{\sqrt{3}}}(x)$ for some $x\in \partial\Omega$. If we assume that $l_0 \leq c\eps$ then by the observation that $\partial\Omega(\eps)$ is $C^{1,1}$-regular with the explicit estimate~\eqref{eq:modcontinuity_eps} we can apply Lemma~\ref{lem:BoundaryEstimate} to the second integrand of~\eqref{eq:int_split}, assuming $c$ is small enough (depending only on dimension). By also applying Lemma~\ref{lem:localBerezin_bulk} to the first integrand in~\eqref{eq:int_split} this yields that
\begin{align}\label{eq:integralform_eps_asymp}
    \Tr(-h^2&\Delta_{\Omega(\eps)}-1)_\limminus 
  \geq
    \LT{1}{n}  \int_{\Omega_*(\eps)}  \int_{\Omega(\eps)} \phi_u^2(x) l(u)^{-n}\, dx du\, h^{-n} \\
    & + \int_{\Omega_b(\eps)} \biggl(\LT{1}{n}\int_{\Omega(\eps)} \phi_u^2(x)\, dx\, h^{-n}  - \frac{1}{4}\LT{1}{n-1}  \int_{\partial \Omega(\eps)} \phi^2_u(x) \, d\sigma(x)\, h^{-n+1}\biggr) l(u)^{-n}\, du\\
    & - c h^{-n+1}\int_{\Omega_b(\eps)} \Bigl(h l(u)^{-2}+l(u)\eps^{-2}+ \eps^{-1}h^{-1}l(u)\Bigr) du
    - c(|\Omega|+|\partial\Omega|)l_0^{-1} h^{-n+2},
\end{align}
where we used the $C^1$-modulus of continuity for $\partial\Omega$ in~\eqref{eq:modcontinuity_eps}, and~\eqref{eq:bulk_integralbound},~\eqref{eq:bdry_integralbound} to bound the error terms coming from Lemmas~\ref{lem:Localization} and~\ref{lem:localBerezin_bulk}.

Using~\eqref{eq:Integral_partition_unity}, and~\eqref{eq:bdry_integralbound} we find that~\eqref{eq:integralform_eps_asymp} implies
\begin{align}\label{eq:asymptotics_eps_reg}
  \Tr(-h^2\Delta_{\Omega(\eps)}-1)_\limminus 
   &\geq 
  \LT{1}{n}|\Omega(\eps)|h^{-n}- \frac{\LT{1}{n-1}}{4}|\partial\Omega(\eps)|h^{-n+1}\\
  &\quad - c(|\Omega|+|\partial\Omega|)(h l_0^{-1}+l_0^2 \eps^{-2}+l_0^2 h^{-1}\eps^{-1})h^{-n+1}\\
  &= 
  \LT{1}{n}|\Omega|h^{-n}- \frac{\LT{1}{n-1}}{4}|\partial\Omega|h^{-n+1}\\
  &\quad + (h l_0^{-1}+l_0^2 \eps^{-2}+l_0^2 h^{-1}\eps^{-1} + h^{-1}\eps^2 + \eps)O(h^{-n+1}),
\end{align}
where we in the second step also use~\eqref{eq:vol_per_control}. The final error term of~\eqref{eq:asymptotics_eps_reg} is uniform on compact subsets of $\K^n$ since this is the case for all the error terms leading up to the estimate.

In the construction above we have required that $h\leq Ml_0$ and $l_0/c \leq \eps < r(\Omega)$, for a dimensional constant $c$, and $l_0< 1$. Setting $l_0 = c h^\alpha$, $M=1/c$, and $\eps = h^\beta$ for some $0<\beta\leq \alpha<1$ we find that our assumptions are satisfied for all $0<h< \min\{1, r(\Omega)^{1/\beta}\}$. With these choices the expression in the parenthesis of the last term in~\eqref{eq:asymptotics_eps_reg} becomes
\begin{equation}
  h l_0^{-1}+l_0^2 \eps^{-2}+l_0^2 h^{-1}\eps^{-1} + h^{-1}\eps^2 + \eps \lesssim h^{1-\alpha}+h^{2\alpha-2\beta}+h^{2\alpha-1-\beta}+h^{2\beta-1}+ h^{\beta}.
\end{equation}
Choosing $\alpha=6/7, \beta=4/7$ we find
\begin{equation}
  h^{1-\alpha}+h^{2\alpha-2\beta}+h^{2\alpha-1-\beta}+h^{2\beta-1}+ h^{\beta} \lesssim h^{1/7}.
\end{equation}

By~\eqref{eq:monotonicity_innerapprox}, and since the the error term in~\eqref{eq:asymptotics_eps_reg} is uniform on compact subsets of $\K^n$, this completes the proof of Lemma~\ref{lem:one_sided_two_term_asymptotics} for $\gamma=1$. As noted above the statement for $\gamma>1$ follows from an application of the Aizenman--Lieb identity. 
\end{proof}

\begin{proof}[Proof of Lemma~\ref{lem:Asymptotic_Polytopes}]
  Fix $\Omega\in \mathcal{P}_m$. By Lemma~\ref{lem:one_sided_two_term_asymptotics} we only need to prove the corresponding upper bound for $\Tr(-h^2\Delta_\Omega-1)_\limminus$. The main idea of the proof is similar to that used above for the regular sets $\Omega(\eps)$. However, since the boundary is now not regular enough to use Lemma~\ref{lem:BoundaryEstimate} close to every point we split our domain of integration into three parts: Define
  \begin{align}
      \Omega^* &:= \{u\in \R^n: \supp \phi_u \cap \Omega \neq \emptyset\},\\
      \Omega_* &:=  \{u \in \R^n: \supp \phi_u \subset \Omega\},\\
      \Omega_b &:= \{u \in \Omega^*: \supp \phi_u \cap \partial\Omega \mbox{ is a piece of a hyperplane}\},\\
      \Omega_s &:= \Omega^*\setminus (\Omega_* \cup \Omega_b).
  \end{align} 
  The set $\Omega^*$ is again the set of $u\in \R^n$ where the localized trace $\Tr(\phi_u(-h^2\Delta_\Omega-1)\phi_u)_\limminus$ is non-zero. The set $\Omega_*$ is the bulk of $\Omega$, where the effect from the boundary is not felt. Finally $\Omega_b$ and $\Omega_s$ are the remaining parts of $\Omega^*$. The first set $\Omega_b$ is where the intersection of $\supp \phi_u$ with the boundary consists of part of a single face of $\Omega$, and hence we can apply Lemma~\ref{lem:BoundaryEstimate} with $\omega\equiv 0$. The second set $\Omega_s$ is where the intersection of $\supp \phi_u$ with the boundary contains pieces of several faces of $\Omega$, we shall show that the contribution from this set is negligible in the limit $h\to 0^\limplus$.

  By Lemma~\ref{lem:Localization} and~\eqref{eq:bulk_integralbound} we have that, for $0<h\leq l_0$,
  \begin{align}
    \Tr(-h^2\Delta_\Omega-1)_\limminus 
    &\leq 
    \int_{\Omega_*}\Tr(\phi_u(-h^2\Delta_\Omega-1)\phi_u)_\limminus l(u)^{-n}\,du\\
    &\quad 
    + \int_{\Omega_b}\Tr(\phi_u(-h^2\Delta_\Omega-1)\phi_u)_\limminus l(u)^{-n}\,du\\
    &\quad 
    + \int_{\Omega_s}\Tr(\phi_u(-h^2\Delta_\Omega-1)\phi_u)_\limminus l(u)^{-n}\,du + c(|\Omega|+|\partial\Omega|) l_0^{-1} h^{-n+2}.
  \end{align}
  We estimate the first and third terms using Lemma~\ref{lem:LocalBerezin}, and apply Lemma~\ref{lem:BoundaryEstimate} with $\omega\equiv 0$ to the integrand of the second, this yields
  \begin{align}
    \Tr(-h^2\Delta_\Omega-1)_\limminus 
    &\leq \LT{1}{n}\int_{\Omega^*}\int_\Omega \phi_u(x)^2 l(u)^{-n}\,dx du\,h^{-n}\\
    &
     - \frac{\LT{1}{n-1}}{4}h^{-n+1}\int_{\Omega_b\cup \Omega_s}\int_{\partial\Omega}\phi_u(x)^2 l(u)^{-n}\,d\sigma(x)du\\
     &
      + \frac{\LT{1}{n-1}}{4}h^{-n+1}\int_{\Omega_s}\int_{\partial\Omega}\phi_u(x)^2 l(u)^{-n}\,d\sigma(x)du + c(|\Omega|+|\partial\Omega|)l_0^{-1}h^{-n+2}.\\[-25pt]
  \end{align}
  Here we have added and subtracted the boundary term integrated over $\Omega_s$, and used~\eqref{eq:bdry_integralbound} to bound the remainder from our application of Lemma~\ref{lem:BoundaryEstimate}. Using~\eqref{eq:Integral_partition_unity} we obtain that
  \begin{align}\label{eq:final_upperbound_polygons}
    \Tr(-h^2\Delta_\Omega-1)_\limminus 
    &\leq \LT{1}{n}|\Omega|h^{-n}- \frac{\LT{1}{n-1}}{4}|\partial\Omega|h^{-n+1}\\
     &
      + \frac{\LT{1}{n-1}}{4}h^{-n+1}\int_{\Omega_s}\int_{\partial\Omega}\phi_u(x)^2 l(u)^{-n}\,d\sigma(x)du + c(|\Omega|+|\partial\Omega|)l_0^{-1}h^{-n+2}.\\[-25pt]
  \end{align}

  Using~\eqref{eq:Assumptions_phi} and the convexity of $\Omega$ it holds that
  \begin{equation}
    \int_{\Omega_s}\int_{\partial\Omega}\phi_u(x)^2 l(u)^{-n}\,d\sigma(x)du \leq c\int_{\Omega_s}|{\supp \phi_u} \cap \partial\Omega| l(u)^{-n}\,du \leq c \int_{\Omega_s}l(u)^{-1}\,du \leq c|\Omega_s|l_0^{-1},
  \end{equation}
  where we used that $l(u)\geq l_0/4$. We want to prove that we can choose $l_0$ such that
  \begin{equation}\label{eq:final_error_polygons}
    h^{-n+1}l_0^{-1}|\Omega_b| + (|\Omega|+|\partial\Omega|)l_0^{-1}h^{-n+2} = o(h^{-n+1})
  \end{equation}
  uniformly for $\Omega$ in compact subsets of $\mathcal{P}_m$. If we can prove that such a choice is possible the combination of~\eqref{eq:final_upperbound_polygons} and Lemma~\ref{lem:one_sided_two_term_asymptotics} implies the claimed asymptotic expansion for $\gamma=1$. As above an application of the Aizenman--Lieb identity completes the proof for all $\gamma>1$.

  Our aim is to show that $|\Omega_b|$ is small, specifically we shall show that it is $\sim l_0^2$. To this end we shall prove that $\Omega_b$ is contained in an $l_0$-neighbourhood of the $(n-2)$-dimensional faces of $\Omega$. 

  Take $u\in \Omega_s$. By definition there are two points $x_1, x_2\in B_{l(u)}(u)\cap \partial\Omega$ such that $x_1, x_2$ belong to two different faces of $\Omega$ (otherwise $u$ would be in $\Omega_b$). Let $x_0$ be a point in $\Omega$ such that $B_{r(\Omega)}(x_0)\subset \Omega$. Consider the plane spanned by the points $x_0, x_1, x_2$, noting that $x_0, x_1, x_2$ cannot lie on a line since by convexity this would imply that $x_0\in \Omega^c$ which is a contradiction. Without loss of generality we can assume that $x_0$ is the origin. Since $|x_1-x_2|\leq 2l_0/{\sqrt{3}}$ we can if $l_0\leq r(\Omega)$ also assume that $x_1, x_2$ are in the same half-plane~$H$. 

  Let $\Omega'$ be the polygon obtained as the intersection of $\Omega$ with this plane. Clearly $r(\Omega')\geq r(\Omega)$ and $D(\Omega')\leq D(\Omega)$. We also note that the segment of $\partial\Omega'\cap H$ connecting $x_1, x_2$ must contain a point belonging to an $(n-2)$-dimensional face of $\Omega$. Let $x'$ be any such point. By convexity $\Omega'$ contains the open triangle which has one vertex at $x'$ and the other two on $\partial B_{r(\Omega)}(x_0)\cap L$, where $L$ is the line through $x_0$ perpendicular to that through $x_0$ and $x'$. In other words we consider the isosceles triangle with one side being a diameter of the disk $B_{r(\Omega)}(x_0)$ and symmetry axis being the segment from $x_0$ to $x'$. As $x_1, x_2\in \partial\Omega$ they are necessarily in the complement of this triangle. Since $|x_1-x_2|\leq 2l_0/{\sqrt{3}}$ and $|x_0-x'|\leq D(\Omega)$ the convexity of $\Omega'$ and elementary trigonometry gives us that 
  \begin{equation}
       \max\{|x_1-x'|, |x_2-x'|\} \leq  \frac{cD(\Omega)}{r(\Omega)} l_0.
   \end{equation}

  We can thus conclude that $\Omega_b$ is contained in a $\tfrac{c D(\Omega)}{r(\Omega)}l_0$-neighbourhood of the $(n-2)$-dimensional faces of $\Omega$. Let $\{F_k\}_k$ denote the collection of these faces. There are fewer than $\binom{m}{2}$ such faces and each of them is contained in a subset of an $(n-2)$-dimensional affine subspace of $\R^n$ whose diameter is less than $D(\Omega)$. Hence we find that
  \begin{align}
    |\Omega_b| &\leq \bigl|\bigl\{u\in \R^n: \dist(u, \cup_{k}F_k)\leq \tfrac{c D(\Omega)}{r(\Omega)}l_0\bigr\}\bigr| \\ 
    &\leq \sum_{k}\bigl|\bigl\{u\in \R^n: \dist(u, F_k) \leq \tfrac{c D(\Omega)}{r(\Omega)} l_0\bigr\}\bigr|\\
    &\leq \binom{m}{2} \bigl|\bigl\{u\in \R^n: \dist(u, \hat F)< \tfrac{c D(\Omega)}{r(\Omega)} l_0\bigr\}\bigr| \leq \tfrac{c D(\Omega)^n}{r(\Omega)^2} l_0^{2},
  \end{align}
  where $\hat F = \{u \in \R^n: u_1=u_2=0, |u|\leq 2 D(\Omega)\}$.

  Returning to~\eqref{eq:final_error_polygons} we can conclude that, with $l_0=h^{1/3}$,
  \begin{equation}
    h^{-n+1}(l_0^{-1}|\Omega_b|+(|\Omega|+|\partial\Omega|)hl_0^{-2}) \leq c h^{-n+4/3}\bigl(\tfrac{D(\Omega)^n}{r(\Omega)^2}+|\Omega|+|\partial\Omega|\bigr).
  \end{equation}
  As the choice of $l_0$ clearly fulfils the requirements $h\leq l_0\leq \min\{1, r(\Omega)\}$ as soon as $h\leq \min\{1, r(\Omega)^3\}$ this completes the proof of Lemma~\ref{lem:Asymptotic_Polytopes}.
\end{proof}


\section{Maximizing Riesz means over disjoint unions of convex domains}\label{sec:Disjoint_unions}
In this section we show that our results are unchanged if one allows also for disjoint unions of convex domains. 
We begin by proving that the result remains true if one allows two convex components.

\begin{lemma}\label{lem:Two_comp}
  Let $\A$ be a closed subset of $\K^n$ which is invariant under dilations and satisfies the assumption in~\ref{itm:Prop_II} of Proposition~\ref{prop:Asymptotic_Shape_Thm}. Fix $\gamma\geq 1$ and let $\Omega_{\specParam, \gamma}(\A^2)$ denote any extremal domain of the shape optimization problem
  \begin{equation}
    \sup\{\Tr(-\Delta_\Omega-\specParam)_\limminus^\gamma: |\Omega|=1,\ \Omega= \Omega_1 \cup \Omega_2,\ \Omega_1\cap \Omega_2 =\emptyset,\ \mbox{and } \Omega_j\in \A \mbox{ or }\Omega_j=\emptyset\}.
  \end{equation}
  Let also $\Omega_{\specParam}^1$ denote the largest of the two components of $\Omega_{\specParam, \gamma}(\A^2)$. 

  For any sequence $\{\specParam_j\}_{j\geq 1}\uparrow \infty$ the corresponding sequence $\{\Omega_{\specParam_j}^1\}_{j\geq 1}$ has a subsequence which, up to rigid transformations, converges in $\A$. Moreover, $\Omega_\infty$ the limit of such a subsequence has unit measure and minimizes the perimeter in $\A$:
  \begin{equation}
    |\partial\Omega_\infty|=\inf\{ |\partial\Omega|: \Omega \in \A, |\Omega|=1\}.
  \end{equation}
\end{lemma}
\begin{proof}[Proof of Lemma~\ref{lem:Two_comp}]
  Fix $\gamma \geq 1$ and let $\Omega_{\specParam,\gamma}(\A^2) = \Omega_\specParam = \Omega^1_\specParam \cup \Omega^2_\specParam$. Assume without loss of generality that $|\Omega^1_\specParam| \geq 1/2$. Since the Riesz mean is additive under disjoint unions the two components must be maximizers for the shape optimization problems among domains in $\A$ of their respective measure. After rescaling to unit measure one finds that $\Omega_\specParam^1$ solves the one-component optimization problem at $\specParam'= \specParam |\Omega_\specParam^1|^{2/n}$. Thus by Proposition~\ref{prop:Asymptotic_Shape_Thm} we are done as soon as we can show that $|\Omega_{\specParam_j}^1|\to 1$ as $j\to \infty$.

  After possibly passing to a subsequence of $\{\specParam_j\}_{j\geq 1}$ we have two possibilities: 

  {\noindent\bf Case 1:} $|\Omega^2_{\specParam_j}|\to 0$ as $j \to \infty$. In which case we are done.\smallskip

  {\noindent\bf Case 2:} $|\Omega_{\specParam_j}^2|\geq c>0$. Since Riesz means are additive under disjoint unions we have that the bound in Theorem~\ref{thm:Improved_Berezin_Convex} holds also in our current setting: Sum the corresponding bounds for the components of the disjoint union. Hence by arguing as in the first part of the proof of Proposition~\ref{prop:Asymptotic_Shape_Thm} we find that $|\partial\Omega_{\specParam_j}|\leq c.$
  Thus both sequences $\{\Omega_{\specParam_j}^1\}_{j\geq 1}, \{\Omega_{\specParam_j}^2\}_{j\geq 1}$ are after translation contained in a compact subset of $\K^n$. Hence our assumptions imply that
    \begin{align}
       \Tr(-\Delta_{\Omega_{\specParam_j}}-\specParam_j)_\limminus^\gamma &= \Tr(-\Delta_{\Omega_{\specParam_j}^1}-\specParam)_\limminus^\gamma+\Tr(-\Delta_{\Omega^2_{\specParam_j}}-\specParam_j)_\limminus^\gamma\\
       &=
        \LT{\gamma}{n} |\Omega_{\specParam_j}|\specParam_j^{\gamma+n/2}- \frac{\LT{\gamma}{n-1}}{4}|\partial\Omega_{\specParam_j}|\specParam_j^{\gamma+ (n-1)/2}+o(\specParam_j^{\gamma+ (n-1)/2}),
     \end{align} 
     as $j\to \infty$.
     Arguing as in the proof of Proposition~\ref{prop:Asymptotic_Shape_Thm} we find that $\Omega_{\specParam_j}$ converges to a domain which minimizes the perimeter among domains with at most two components, each of which is in $\A$. If $\Omega'=\Omega_1' \cup \Omega_2'$ with $\Omega_j'\in \A$ it is clear that the perimeter of $\Omega'$ is minimal when the perimeter of the two components are minimizers of the perimeter in $\A$ among sets of their respective measure. By scaling we find that $|\partial\Omega'|=(|\Omega_1'|^{(n-1)/n}+|\Omega_2'|^{(n-1)/n})\inf\{|\partial\Omega|: \Omega\in \A, |\Omega|=1\}$. Since $\eta^{(n-1)/n}+(1-\eta)^{(n-1)/n}\geq 1$ with equality if and only if $\eta=0$ or $1$ we find that any domain minimizing the perimeter must have only one component. This contradicts the assumption that $|\Omega_{\specParam_j}^2|\geq c$, and hence completes the proof.
\end{proof}

Using the same idea as above it is not difficult to prove the corresponding result when any fixed and finite number of components is allowed. However, our goal is here to show that this restriction is in fact not necessary and we can allow for an arbitrary number of components. The only reason to first prove the two-component case is that it will be used in the proof of the general result.

\begin{corollary}\label{cor:Disjoint_unions}
  Let $\A$ be a closed subset of $\K^n$ which is invariant under dilations and satisfies the assumption in~\ref{itm:Prop_II} of Proposition~\ref{prop:Asymptotic_Shape_Thm}. Fix $\gamma\geq 1$ and let $\Omega_{\specParam, \gamma}(\A^\infty)$ denote any extremal domain of the shape optimization problem
  \begin{equation}
    \sup\{\Tr(-\Delta_\Omega-\specParam)_\limminus^\gamma: |\Omega|=1,\ \Omega= \textstyle{\bigcup_{k\geq 1}} \Omega_k,\ \Omega_k \in \A,\ \Omega_k\cap \Omega_{k'} =\emptyset \mbox{ if } k\neq k'\}.
  \end{equation}
  Let also $\Omega_{\specParam}^1$ denote the largest of the components of $\Omega_{\specParam, \gamma}(\A^\infty)$. 

  For any sequence $\{\specParam_j\}_{j\geq 1}\uparrow \infty$ the corresponding sequence $\{\Omega_{\specParam_j}^1\}_{j\geq 1}$ has a subsequence which, up to rigid transformations, converges in $\A$. Moreover, $\Omega_\infty$ the limit of such a subsequence has unit measure and minimizes the perimeter in $\A$:
  \begin{equation}
    |\partial\Omega_\infty|=\inf\{ |\partial\Omega|: \Omega \in \A, |\Omega|=1\}.
  \end{equation}
\end{corollary}

\begin{remark}
  We note that Corollary~\ref{cor:Disjoint_unions} can be interesting even in extremely simple cases. For instance, it implies that among unions of disjoint balls the maximizers will as $\specParam\to \infty$ converge to a \emph{single} ball of unit measure.
\end{remark}

\begin{proof}[Proof of Corollary~\ref{cor:Disjoint_unions}]
  Again we can argue as in Proposition~\ref{prop:Asymptotic_Shape_Thm} to find that
  \begin{equation}\label{eq:multicomp_perbound}
    |\partial\Omega_{\specParam, \gamma}(\A^\infty)|\leq c.
  \end{equation} 
  Moreover, by Faber--Krahn's inequality we know that each component of a maximizer $\Omega_{\specParam, \gamma}(\A^\infty)$ has measure greater than $c\specParam^{-n/2}$. Indeed, the Riesz mean is zero for any component with smaller measure, which contradicts the maximality of $\Omega_{\specParam,\gamma}(\A^\infty)$ since we can remove such components and rescale the remaining domain to have measure one and in the process increasing the Riesz mean.

  Let $\Omega_{\specParam_j,\gamma}(\A^\infty)= \bigcup_{k\geq1} \Omega_{\specParam}^k$ be a maximizer, where we assume $|\Omega_{\specParam_j}^k|\geq |\Omega_{\specParam_j}^{k'}|$ if $k<k'$. Fix $\{\specParam_j\}_{j\geq 1} \uparrow \infty$. After possibly passing to a subsequence we can assume that $|\Omega_{\specParam_j}^1|< 1-\eps$, for some $\eps>0$. If this is not the case we are already done. 
  \medskip
  
  {\noindent\bf Step 1:} We first exclude that all components have size $\sim \specParam_j^{-n/2}$. Assume that along the sequence $\specParam_j$ (or a subsequence thereof) we have that $|\Omega_{\specParam_j}^1|\leq c\specParam_j^{-n/2}$ for some $c>0$. Due to the measure constraint we must have $\sim\specParam_j^{n/2}$ components. By the isoperimetric inequality
  \begin{equation}
     |\partial\Omega_{\specParam_j}| = \sum_{k\geq 1}|\partial\Omega_{\specParam_j}^k| \gtrsim \specParam_j^{n/2} \specParam_j^{-(n-1)/2} =\specParam_j^{1/2}\to \infty,
   \end{equation} 
   which contradicts~\eqref{eq:multicomp_perbound}.
  \medskip
  
  {\noindent\bf Step 2:}
   The set $\Omega_{\specParam_j}^1\! \cup \Omega_{\specParam_j}^2$ is a maximizer for the problem
   \begin{equation}
     \sup\{\Tr(-\Delta_\Omega-\specParam_j)_\limminus^\gamma : |\Omega|=m_j,\ \Omega= \Omega_1 \cup \Omega_2,\ \Omega_1\cap \Omega_2 =\emptyset,\ \mbox{and } \Omega_j\in \A \mbox{ or }\Omega_j=\emptyset\},
   \end{equation}
   with $m_j=|\Omega_{\specParam_j}^1|+|\Omega_{\specParam_j}^2|$.
  By Step 1 we can assume that $m_j \specParam_j^{n/2}\to \infty$ and hence this problem is, after rescaling to unit measure, equivalent to that considered in Lemma~\ref{lem:Two_comp}. Hence we find that for $j$ large enough $|\Omega_{\specParam_j}^2|\leq c|\Omega_{\specParam_j}^1|$ for any $0<c<1$ to be chosen later.
  \medskip
  
  {\noindent\bf Step 3:} Similarly, $\widehat\Omega_1=\bigcup_{k\geq 2}\Omega_{\specParam_j}^k$ is a maximizer for the problem
   \begin{equation}
     \sup\{\Tr(-\Delta_\Omega-\specParam_j)_\limminus^\gamma : |\Omega|=|\widehat\Omega_1|,\ \Omega= \textstyle{\bigcup_{k\geq 1}} \Omega_k,\ \Omega_k \in \A,\ \Omega_k\cap \Omega_{k'} =\emptyset \mbox{ if } k\neq k'\}.
   \end{equation}
    Since $|\widehat\Omega_1|=1-|\Omega_{\specParam_j}^1|>\eps$ this problem is again in the asymptotic regime and we can argue as in Steps $1$ and $2$ and find that $|\Omega_{\specParam_j}^3|\leq c|\Omega_{\specParam_j}^2|$ for any $0<c<1$ if $j$ is large enough. 
  \medskip
  
  {\noindent\bf Step 4:} Set $c=\tfrac{\eps}{2-\eps}$. We can then iterate the arguments above: For each $l> 1$ we have $|\bigcup_{k\geq l}\Omega_{\specParam_j}^k| = 1- |\bigcup_{k=1}^{l-1}\Omega_{\specParam_j}^k|\geq 1-|\Omega_{\specParam_j}^1|\sum_{k=1}^{l-1}\tfrac{\eps^{k-1}}{(2-\eps)^{k-1}}> \eps/2$, this ensures that the maximization problem which $\widehat \Omega_{l} = \bigcup_{k\geq l}\Omega_{\specParam_j}^{k}$ solves is still in the asymptotic regime when $j\to \infty$. Hence, by Steps $1$-$3$ it holds that $|\Omega_{\specParam_j}^{k+1}|\leq |\Omega^{k}_{\specParam_j}|\frac{\eps}{2-\eps}$, for all $k\geq 1$, provided that $\specParam_j$ is large enough (depending only on $\eps$).
\medskip
  
  {\noindent\bf Step 5:} Calculate the measure of $\Omega_{\specParam_j,\gamma}(\A^\infty)$:
   \begin{equation}
     |\Omega_{\specParam_j,\gamma}(\A^\infty)|=\sum_{k\geq 1}|\Omega_{\specParam_j}^k|\leq |\Omega_{\specParam_j}^1|\sum_{k=1}^\infty \frac{\eps^{k-1}}{(2-\eps)^{k-1}} \leq (1-\eps)\sum_{k=1}^\infty \frac{\eps^{k-1}}{(2-\eps)^{k-1}} = 1-\eps/2,
   \end{equation}
   which is a contradiction for all $\eps>0$ and hence $|\Omega_{\specParam_j}^1|\to 1$ as $j\to \infty$. 
\medskip

   By Proposition~\ref{prop:Asymptotic_Shape_Thm} we can conclude that $\Omega_{\specParam_j}^1$ converges to a domain which minimizes the perimeter among domains of unit measure in $\A$. This completes the proof.
\end{proof}

\noindent{\bf Acknowledgements.} It is a pleasure to thank Ari Laptev for his encouragement and support. 
The author is deeply indebted to the anonymous referees whose comments and suggestions improved both the exposition and results immensely. The author would also like to thank Rupert Frank, Richard Laugesen and Timo Weidl for many helpful suggestions and discussions. Financial support from the Swedish Research Council grant no.\ 2012-3864 and Svenska Matematikersamfundet is gratefully acknowledged.


\bibliographystyle{amsplain}

\def\myarXiv#1#2{\href{http://arxiv.org/abs/#1}{\texttt{arXiv:#1\,[#2]}}}

\end{document}